\documentclass{amsart}
\usepackage{etex}
\usepackage{amsmath,amsthm,amsfonts,amscd,amssymb,mathtools} 
\usepackage[normalem]{ulem}
\usepackage{verbatim}      	
\usepackage{url}		
\usepackage{epic,eepic,latexsym}
\usepackage{graphicx}
\usepackage{color}
\usepackage{lcd}
\usepackage{mathrsfs}
\usepackage[centering,text={15.5cm,22cm},
        marginparwidth=20mm,dvips]{geometry}
\usepackage{verbatim}
\usepackage[linktocpage]{hyperref}
\usepackage{lscape}
\usepackage{tabularx}
\usepackage{marginnote}
\usepackage{wrapfig,caption}
\usepackage[all,color]{xy}
\UseCrayolaColors
\usepackage{enumitem}
\usepackage{tikz-cd}
\usepackage{stmaryrd}
\usepackage{textcomp}

\DeclareMathOperator{\Hom}{Hom}

\DeclareMathOperator{\codim}{codim}

\DeclareMathOperator{\Spec}{Spec}
\DeclareMathOperator{\id}{id}
\DeclareMathOperator{\Id}{id}

\DeclareMathOperator{\Tr}{Tr}

\DeclareMathOperator{\val}{val}

\DeclareMathOperator{\sign}{sgn}

\DeclareMathOperator{\trop}{trop}

\DeclareMathOperator{\out}{out}

\DeclareMathOperator{\Mult}{Mult}

\DeclareMathOperator{\Aut}{Aut}

\DeclareMathOperator{\SH}{SH}

\DeclareMathOperator{\Ad}{Ad}
\DeclareMathOperator{\ad}{ad}

\DeclareMathOperator{\LL}{L}
\DeclareMathOperator{\ov}{ov}
\DeclareMathOperator{\forget}{Forget}

\DeclareMathOperator{\GW}{GW}

\DeclareMathOperator{\inde}{index}
\DeclareMathOperator{\domain}{domain}

\DeclareMathOperator{\UnShuff}{UnShuff}

\DeclareMathOperator{\Trop}{Trop2Cob}
\DeclareMathOperator{\Poly}{PolyVect}

\DeclareMathOperator{\Mod}{Mod}
\DeclareMathOperator{\Cob}{Cob}
\DeclareMathOperator{\even}{even}
\DeclareMathOperator{\topp}{top}

\DeclareMathOperator{\Flags}{Flags}

\let\?=\overline
\let\:=\colon
\let\llb=\llbracket
\let\rrb=\rrbracket
\let\bb=\mathbb

\let\rar=\rightarrow

\let\ra=\rightarrow

\let\f=\mathfrak
\let\s=\mathcal

\let\wh=\widehat
\let\wt=\widetilde

\def\risom{\buildrel\sim\over{\smashedlongrightarrow}}
 \def\smashedlongrightarrow{\setbox0=\hbox{$\longrightarrow$}\ht0=1.25pt\box0}
 
 \def\edge{\mbox{\hspace{-.15 cm}\textminus\hspace{-.05 cm}\textminus\hspace{-.05 cm}\textminus\hspace{-.05 cm}\textminus}\hspace{-.15 cm}}
 
 \def\leftsink{\prec\mbox{\hspace{-.25 cm}\textminus\hspace{-.05 cm}\textminus\hspace{-.05 cm}\textminus\hspace{-.05 cm}\textminus}}
 
\def\leftsource{\succ\mbox{\hspace{-.15 cm}\textminus\hspace{-.05 cm}\textminus\hspace{-.05 cm}\textminus}}

\def\rightsink{\mbox{\textminus\hspace{-.05 cm}\textminus\hspace{-.05 cm}\textminus\hspace{-.05 cm}\textminus}\hspace{-.25 cm}\succ}

\def\rightsource{\mbox{\textminus\hspace{-.05 cm}\textminus\hspace{-.05 cm}\textminus}\hspace{-.15 cm}\prec}

\def\lrsource{\succ\mbox{\hspace{-.15 cm}\textminus\hspace{-.05 cm}\textminus\hspace{-.05 cm}\textminus}\hspace{-.15 cm}\prec}

\def\lrsink{\prec\mbox{\hspace{-.25 cm}\textminus\hspace{-.05 cm}\textminus\hspace{-.05 cm}\textminus\hspace{-.05 cm}\textminus}\hspace{-.25 cm}\succ}

\def\lsourcersink{\succ\mbox{\hspace{-.15 cm}\textminus\hspace{-.05 cm}\textminus\hspace{-.05 cm}\textminus}\hspace{-.25 cm}\succ}

\def\lsinkrsource{\prec\mbox{\hspace{-.25 cm}\textminus\hspace{-.05 cm}\textminus\hspace{-.05 cm}\textminus}\hspace{-.15 cm}\prec}

\newcommand {\kk} {\Bbbk}
\newcommand {\A} {{\bf A}}

\newcommand {\RR} {{\mathbb{R}}}
\newcommand {\ZZ} {{\mathbb{Z}}}

\newcommand {\shM} {{\mathcal{M}}}

\theoremstyle{plain}
 \newtheorem{thm}{Theorem}[section]
 
 \newtheorem{lem}[thm]{Lemma}
  \newtheorem{prop}[thm]{Proposition}
  
   \newtheorem{cor}[thm]{Corollary}
	 \newtheorem{obs}[thm]{Observation}

\theoremstyle{definition}
 \newtheorem{dfn}[thm]{Definition}
  
 \newtheorem{ntn}[thm]{Notation}
 \newtheorem{eg}[thm]{Example}
 
 \newtheorem{lemdfn}[thm]{Lemma/Definition}

\theoremstyle{remark} 
 \newtheorem{rmk}[thm]{Remark}

\title[Tropical QFT, mirror polyvector fields, and multiplicities of tropical curves]{Tropical quantum field theory, mirror polyvector fields, and multiplicities of tropical curves}
\author{Travis Mandel}
\address{Department of Mathematics\\
University of Oklahoma\\
Norman, OK 73019\\
USA}
\email{tmandel{\char'100}ou.edu}
\author{Helge Ruddat}
\address{JGU Mainz \& Univ. Hamburg}
\thanks{The first author was supported by the National Science Foundation RTG Grant DMS-1246989, and later by the Starter Grant ``Categorified Donaldson-Thomas Theory'' no. 759967 of the European Research Council.  The second author was supported by the DFG Emmy-Noether grant RU 1629/4-1 and is grateful for hospitality at the IAS in Princeton.}
\email{ruddat@uni-mainz.de}

\begin{document}

\begin{abstract}
We introduce algebraic structures on the polyvector fields of an algebraic torus that serve to compute multiplicities in tropical and log Gromov-Witten theory while also connecting to the mirror symmetry dual deformation theory of complex structures. Most notably these structures include a tropical quantum field theory and an $L_{\infty}$-structure. The latter is an instance of Getzler's gravity algebra, and the $l_2$-bracket is a restriction of the Schouten-Nijenhuis bracket.
We explain the relationship to string topology in the appendix (thanks to Janko Latschev).
\end{abstract}

\maketitle

\setcounter{tocdepth}{1}
\tableofcontents  

\section{Introduction}
Counts of tropical curves have been identified to govern the infinitesimal smoothing of mirror dual maximally degenerate complex Calabi-Yau varieties in \cite{GS11,GPS,KLM}. 
More recently, the importance of Batalin-Vilkovisky (BV) structures has emerged in this context \cite{CLM,FFR,Felten}.  We find a direct relationship between these perspectives by expressing multiplicities of tropical curves in terms of certain iterated higher brackets of polyvector fields which coincide with Getzler's gravity algebra operations \cite{Getzler} and with Chas-Sullivan's $L_\infty$-structure on equivariant string homology \cite{CS}.  The smoothing algorithm in \cite{GS11} can be interpreted as gluing infinitesimal versions of localizations of algebraic tori along particular isomorphisms called \emph{wall-crossing transformations}.  We find that these wall-crossing transformations, along with their action on rational functions and the induced action on polyvector fields, are consistent with our description of tropical multiplicities.

In any theorem relating counts of tropical curves to Gromov-Witten invariants, the tropical curves must be counted with certain multiplicities.  
For a planar tropical curve, one may associate a multiplicity to each vertex \cite{Mi} so that the entire multiplicity is simply the product of the multiplicities of the vertices.  But in higher dimensions, or even in two dimensions with psi-class conditions, one lacks such a local description, and the multiplicities are instead given as the index of a complicated map of lattices, see 
\cite[Prop. 5.7]{NS} and for a generalization to psi-classes and boundary conditions \cite{MRud}.  For the aforementioned connection to mirror dual structures, the global descriptions of multiplicities are impractical to work with.  In the present paper, we prove several new formulae for tropical multiplicities in terms of local computations controlled algebraic structures on polyvector fields:
\begin{enumerate}

   \item for arbitrary genus in terms of a two-dimensional tropical quantum field theory (TrQFT), as defined and developed in \S \ref{SectionTrQFT}, see Theorem \ref{MainThm}. Roughly,
a 2D TrQFT is a functor from a category whose objects are tropical degrees and whose morphisms are tropical cobordisms, see Def.~\ref{TropCobDfn}. We prove an algebraic characterization of 2D TrQFT's
in Theorem~\ref{TrQFT-Equivalent-Data}, generalizing the relationship of 2D TQFT's with Frobenius algebras;

   \item for genus zero as the result of an iterated bracket of polyvector fields, see Theorem \ref{BracketMult}. The $2$-bracket $l_2$ here agrees with the Schouten-Nijenhuis bracket, while the higher brackets $l_k$, which appear when $\psi$-classes are present, form an $L_{\infty}$-algebra on the kernel of the BV-operator on polyvector fields.  The reader hoping to extract a practical tropical multiplicity formula (especially in the context of the Gross-Siebert program) is encouraged to read \S\ref{MirrorPolySection} for a concise and self-contained description;
   
    \item for genus zero as a product of vertex multiplicities divided by a product of edge multiplicities, cf. Corollary \ref{SplitCor} and Theorem \ref{SplitThm}.
    
\end{enumerate}
We explain the relevance of the new formulae in the following sections.

\subsection{Theta functions on cluster varieties} 
\label{cluster-l-infty}
In \S \ref{MSapps} we consider Theorem~\ref{BracketMult} in the context of the Gross-Siebert mirror symmetry program \cite{GSprogram,GS11}.  
We show in Proposition~\ref{dl2} that the induced action of a wall-crossing transformation on polyvector fields
agrees with the adjoint action for the Schouten-Nijenhuis bracket $l_2$.  
We conjecture the existence of ``theta polyvector fields'' which extend the notion of theta functions studied by Gross, Hacking, Keel, Kontsevich, and Siebert \cite{CPS,GHK1,GHKK,GHS}.  The first author \cite{ManFrob} used Theorem~\ref{BracketMult} to prove that one can express the \cite{GHKK} theta bases in terms of mirror descendant log Gromov-Witten numbers, see Example~\ref{ThetaEx}.  Proposition~\ref{dl2} suggests that a similar argument might apply to the conjectural theta polyvector fields. Throughout \S\ref{MirrorPolySection}-\S\ref{Linf}, we point out several remarkable connections to \cite{BK} that we don't yet fully understand the significance of.

\hbox{\qquad\qquad\qquad\qquad\qquad\qquad\qquad\qquad\qquad\qquad\qquad\qquad}\vspace{-.5cm}
\begin{wrapfigure}[11]{r}{0.4\textwidth}
\captionsetup{width=.89\linewidth}
\begin{center}\vspace{-.96cm}
{\includegraphics[width=0.327\textwidth]{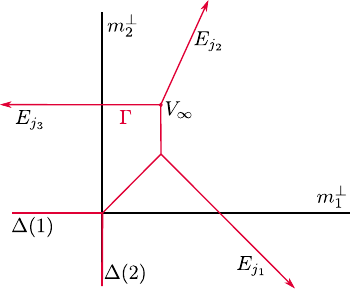}}
\end{center}
\caption{A tropical curve $\Gamma$ contributing to the constant term in the product of three theta functions.}
\label{Fig: theta-product}
\end{wrapfigure}
The following example illustrates how the multiplication rule of theta functions on cluster varieties is determined by means of our tropical multiplicity formulae \cite[Thm.~3.9]{Man3}. The example also introduces the aforementioned $L_\infty$-structure.

\begin{eg}\label{ThetaEx}
Let $\omega(\cdot,\cdot)$ denote an integral skew-symmetric bilinear form on $N:=\ZZ^n$.  Let $\Delta:I\sqcup J \sqcup \{\infty\} \rar N$ be a map of sets with $\Delta$ nonzero on $I\sqcup J$ and with $\Delta(\infty)=0$.  
We require that $m_i:=\omega(\Delta(i),\cdot)\in M:=\Hom(N,\ZZ)$ is nonzero for each $i\in I$.  
A tropical curve $\Gamma$ has degree $\Delta$ if it comes with a bijection of its unbounded edges with $I\sqcup J \sqcup \{\infty\}$ so that the edge $E_i$ corresponding to $i\in I\sqcup J\sqcup \{\infty\}$ has weighted tangent direction $\Delta(i)$. 
For such a curve, we impose conditions on the position of the edges, namely edge $E_i$ is required to be contained in $A_i$ where
\begin{tabbing}
$A_{\infty}$\qquad\qquad\qquad  \= is a point in general position,\\
$A_j$  for $j\in J$\> is $N_{\bb{R}}:=N\otimes_\ZZ\RR$, so no condition on $E_j$,\\
$A_i$ for $i\in I$ \> is a translate of $m_i^{\perp}$ into a general position.
\end{tabbing}
The conditions imply that the edge $E_\infty$ gets contracted under the map that takes $\Gamma$ to $N_\RR:=N\otimes_\ZZ\RR$ while the other unbounded edges are non-contracted.
The $A_i$ with $i\in I$ carry a weight\footnote{Such weights show up, for example, in \cite[\S 4.3]{MRud}.}   $w(A_i)$ equal to the index of $m_i$.
We furthermore require that the unique vertex $V_{\infty}$ of $E_{\infty}$ is $(|J|+1)$-valent, this condition is referred to as $\Psi$.
Figure~\ref{Fig: theta-product} illustrates an example of the image of such a curve $\Gamma$ for $n=2$, $|I|=2$ and $|J|=3$.

One can show that a tropical curve $\Gamma$ of genus $0$ and degree $\Delta$ satisfying $\A:=\{A_i\}$ and $\Psi$ is rigid which means that there don't exist any continuous deformations of the curve which still satisfy the conditions.
Furthermore, for such a curve $\Gamma$, each component of $\Gamma\setminus \{E_{\infty}\}$ contains exactly one of the edges of the form $E_j$ for $j\in J$. 
In \cite{Man3}, the edges indexed by $J$ are indexing \emph{theta functions} whereas the edges indexed from $I$ correspond to \emph{Maslov index zero disks} that originate in the \emph{walls} given by the $m_i^\perp$, see Figure~\ref{Fig: theta-product}.  

We now compute the multiplicity of $\Gamma$ using Theorem \ref{BracketMult}. One can show that $\Gamma$ is trivalent away from $V_\infty$, so we assume this from now on.  
Consider the exterior algebra of polyvector fields on the algebraic torus $\bb{G}_m(M)=\Spec \bb{Z}[N]$ given by
$$A:=\bigwedge^\bullet\Theta_{\Spec\ZZ[N]/\Spec\ZZ}= \ZZ[N]\otimes\Lambda^\bullet M.$$ 
We are going to associate an element $\zeta_{E}\in A$ to each edge $E$ of $\Gamma$ inductively, starting with the unbounded edges by setting
\begin{align*}
\zeta_{E_{\infty}}&:= 1\otimes(\hbox{a generator of }\Lambda^{\topp} M),\\
\zeta_{E_j}&:=z^{\Delta(j)}\otimes 1\qquad\qquad \hbox{for each }j\in J,\\
\zeta_{E_i}&:=z^{\Delta(i)}\otimes m_i \qquad\quad\   \hbox{for each }i\in I.
\end{align*}
We use the $\bb{Z}$-linear map $\ell_1:A\rar A$ given by\footnote{Here, $\iota_n$ denotes the tensor contraction. For $m\in M$, $\iota_n(m)$ is the dual pairing $\langle n,m\rangle$.  More general contractions can be computed with the Leibniz rule $\iota_n(m_0\wedge \cdots \wedge m_k)=\sum_{i=0}^k (-1)^{i} \iota_n(m_i) m_0\wedge \cdots \wedge m_{i-1}\wedge m_{i+1} \wedge \cdots \wedge m_k$.}
$\ell_1(z^n\alpha)=z^n\iota_n(\alpha)$, and define $\ell_k:A^{\otimes k}\rar A$ by
\begin{align}
\label{eq-l-infty}
 \ell_k(z^{n_1}\alpha_1,\ldots,z^{n_k}\alpha_k):= z^{n_1+\ldots+n_k} \iota_{n_1+\ldots+n_k}(\alpha_1\wedge \cdots \wedge \alpha_k).
\end{align}
We view $\Gamma$ as a tree with root $V_{\infty}$ and consider the natural flow=(orientation of edges) from the leaves=(unbounded edges) towards $V_{\infty}$.
At every vertex $V$ that is different from $V_{\infty}$, we have two incoming edges $E_1, E_2$. Let us assume that the associated forms for these edges take the shape $\zeta_{E_k}=z^{n_{E_k}}\otimes \omega(n_{E_k},\cdot)$.  
Then to the outgoing edge $E_3$ we are going to associate
\begin{align*}
\zeta_{E_3}&:=\ell_2(\zeta_{E_1},\zeta_{E_2})=\ell_1\Big(z^{n_{E_1}+n_{E_2}}\otimes\big( \omega(n_{E_1},\cdot)\wedge \omega(n_{E_2},\cdot)\big)\Big)\\
&=\omega(n_{E_1},n_{E_2})z^{n_{E_1}+n_{E_2}}\otimes \omega(n_{E_1}+n_{E_2},\cdot)=\omega(n_{E_1},n_{E_2})z^{n_{E_3}}\otimes \omega(n_{E_3},\cdot)
\end{align*}
with $n_{E_3}=n_{E_1}+n_{E_2}$. Note that, up to an integer multiple, the outgoing edge is again of the shape that we assumed for the incoming edges.
If, instead, $\zeta_{E_2}$ were just equal to $z^{n_{E_2}}$, with $\zeta_{E_1}$ as before, then we find by the same rule
\begin{align*}
    \zeta_{E_3}:=\ell_2(\zeta_{E_1},\zeta_{E_2})=\omega(n_{E_1},n_{E_2})z^{n_{E_3}}.
\end{align*}
If two unbounded edges meet in a trivalent vertex, we are in one of the two above situations because by rigidity the case of two edges $E_i$ with $i\in I$ meeting is excluded. 
For the trivalent vertex $V$, let us define the integer 
$$\Mult(V):=|\omega(n_{E_1},n_{E_2})|$$
which is up to sign the integer coefficient of the outgoing edge in either of the two cases.
By induction along the flow, we produce a form $\zeta_E\in A$ for every edge $E$ of $\Gamma$. 
When all branches of the flow finally reach the sink $V_\infty$, the product of the forms $\zeta_E$ for the edges adjacent to $V_\infty$ is of the form $a z^n\otimes \alpha $ where $a\in\ZZ$ is the product of the terms $\omega(n_{E_1},n_{E_2})$ that we gather along the flow, the exponent $n$ is zero by the balancing condition of $\Gamma$ and the form $\alpha$ is the wedge product of all forms of the edges adjacent to $V_\infty$. 
All non-contracted edge however carry a form in $\Lambda^0 M=\ZZ$ (because the conditions are non over-determining $\Gamma$)
while the contracted edge $E_\infty$ by assumption carries $\Omega$, a generator of $\Lambda^{\topp} M$.
Setting $\Mult(V_\infty)=1$, we therefore conclude the equality
$$
 \prod_{E\ni V_\infty} \zeta_E = \left(\prod_{V} \Mult(V)\right)\otimes\Omega
$$
Our main result about multiplicities from from flows, Theorem \ref{BracketMult}, now gives the multiplicty of $\Gamma$ as a product over vertex multiplicities:
\begin{align}\label{ProdMultV}
    \Mult(\Gamma)=\left(\prod_{V} \Mult(V)\right)\otimes\Omega
\end{align}
The application of Theorem \ref{BracketMult} to this example is used in \cite{ManFrob} to relate the multiplicities of \cite[\S 3.1.2]{Man3} to those of \cite{MRud}, thus relating theta functions on cluster varieties to descendant log Gromov-Witten invariants.

In the quantum version of \cite[Thm. 3.9]{Man3}, the analogous multiplicities correspond to a refinement as in \cite{BG}.  
That is, one defines the quantum multiplicities $\Mult_q(\Gamma)$ by replacing each $\Mult(V)$ in \eqref{ProdMultV} with a Laurent polynomial $\Mult_q(V)$. For $V\neq V_{\infty}$, $\Mult_q(V):=q^{\Mult(V)}-q^{-\Mult(V)}$, while $\Mult_q(V_{\infty}):=q^{\sum_{i<j} \omega(n_{E_i},n_{E_j})}$, where $E_1,\ldots,E_{|J|}$ are the edges (other than $E_{\infty}$) containing $V_{\infty}$, ordered according to the order of the theta function multiplication.  
We expect that \cite{Mikq}'s interpretation of refined counts of planar tropical curves can be generalized to relate the refined tropical counts to sign-weighted counts of real curves or holomorphic disks with boundary on the real locus. The computation here will then imply that the real curve counts determine the holomorphic curve counts. 

We remark here that the computation of multiplicities via iterated Lie brackets 
in the context of theta functions and scattering diagrams
has recently been related to a technique for solving the Maurer-Cartan equation via certain sums over trees, see \cite{LMY}.
\end{eg}

\subsection{Tropical invariance and the Jacobi identity}

The tropical Gromov-Witten numbers are invariant under generic translations of the incidence conditions $\A$.  This of course follows from the fact that these numbers are known to correspond to descendant log Gromov-Witten invariants \cite{MRud}.  On the other hand, a direct proof of this tropical invariance in $2$-dimensional cases (without $\psi$-classes) was given by Gathmann-Markwig in \cite{GM}.  Building off their approach, Figure \ref{JacobiFig} demonstrates that, when multiplicities are computed in terms of Schouten-Nijenhuis brackets as in Theorem \ref{BracketMult}, the invariance of the genus $0$ tropical counts is related to the Jacobi identity.   
\begin{figure}[htb]
    \begin{tabular}{c c c c c}
    \def\svgwidth{100pt}
\begingroup%
  \makeatletter%
  \providecommand\color[2][]{%
    \errmessage{(Inkscape) Color is used for the text in Inkscape, but the package 'color.sty' is not loaded}%
    \renewcommand\color[2][]{}%
  }%
  \providecommand\transparent[1]{%
    \errmessage{(Inkscape) Transparency is used (non-zero) for the text in Inkscape, but the package 'transparent.sty' is not loaded}%
    \renewcommand\transparent[1]{}%
  }%
  \providecommand\rotatebox[2]{#2}%
  \newcommand*\fsize{\dimexpr\f@size pt\relax}%
  \newcommand*\lineheight[1]{\fontsize{\fsize}{#1\fsize}\selectfont}%
  \ifx\svgwidth\undefined%
    \setlength{\unitlength}{353.80989393bp}%
    \ifx\svgscale\undefined%
      \relax%
    \else%
      \setlength{\unitlength}{\unitlength * \real{\svgscale}}%
    \fi%
  \else%
    \setlength{\unitlength}{\svgwidth}%
  \fi%
  \global\let\svgwidth\undefined%
  \global\let\svgscale\undefined%
  \makeatother%
  \begin{picture}(1,0.74186581)%
    \lineheight{1}%
    \setlength\tabcolsep{0pt}%
    \put(0,0){\includegraphics[width=\unitlength,page=1]{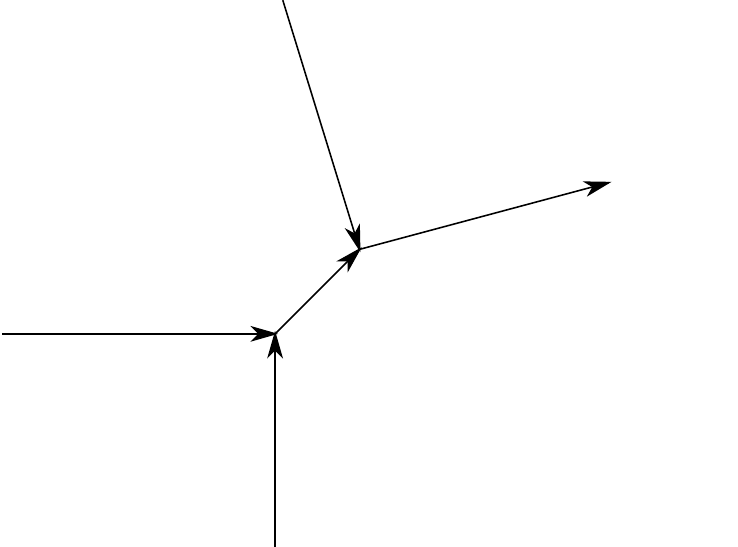}}%
    \put(0.27222946,0.73742634){\color[rgb]{0,0,0}\makebox(0,0)[lt]{\begin{minipage}{0.37365748\unitlength}\raggedright $E_1$\end{minipage}}}%
    \put(-0.00641905,0.27834314){\color[rgb]{0,0,0}\makebox(0,0)[lt]{\begin{minipage}{0.37365748\unitlength}\raggedright $E_2$\end{minipage}}}%
    \put(0.3775868,0.09816022){\color[rgb]{0,0,0}\makebox(0,0)[lt]{\begin{minipage}{0.37365748\unitlength}\raggedright $E_3$\end{minipage}}}%
    \put(0.73635206,0.39080163){\color[rgb]{0,0,0}\makebox(0,0)[lt]{\lineheight{1.25}\smash{\begin{tabular}[t]{l}$E_4$\end{tabular}}}}%
  \end{picture}%
\endgroup%

    & ~ &\def\svgwidth{100pt} 
\begingroup%
  \makeatletter%
  \providecommand\color[2][]{%
    \errmessage{(Inkscape) Color is used for the text in Inkscape, but the package 'color.sty' is not loaded}%
    \renewcommand\color[2][]{}%
  }%
  \providecommand\transparent[1]{%
    \errmessage{(Inkscape) Transparency is used (non-zero) for the text in Inkscape, but the package 'transparent.sty' is not loaded}%
    \renewcommand\transparent[1]{}%
  }%
  \providecommand\rotatebox[2]{#2}%
  \newcommand*\fsize{\dimexpr\f@size pt\relax}%
  \newcommand*\lineheight[1]{\fontsize{\fsize}{#1\fsize}\selectfont}%
  \ifx\svgwidth\undefined%
    \setlength{\unitlength}{359.59552515bp}%
    \ifx\svgscale\undefined%
      \relax%
    \else%
      \setlength{\unitlength}{\unitlength * \real{\svgscale}}%
    \fi%
  \else%
    \setlength{\unitlength}{\svgwidth}%
  \fi%
  \global\let\svgwidth\undefined%
  \global\let\svgscale\undefined%
  \makeatother%
  \begin{picture}(1,0.69900286)%
    \lineheight{1}%
    \setlength\tabcolsep{0pt}%
    \put(0,0){\includegraphics[width=\unitlength,page=1]{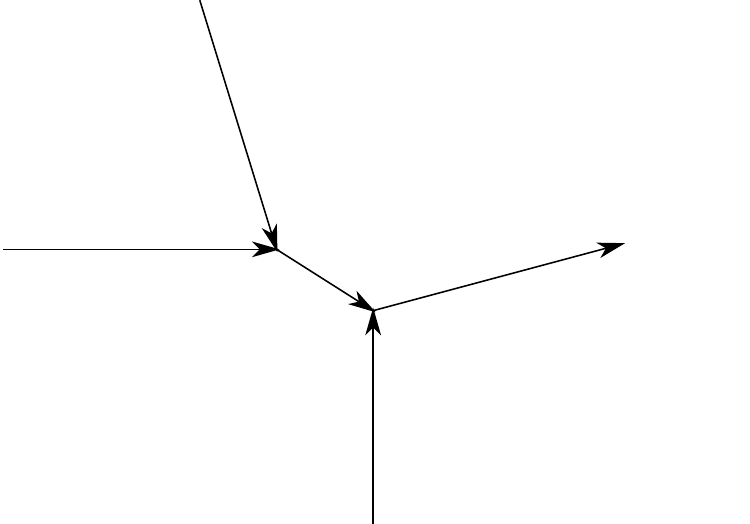}}%
    \put(0.15283929,0.69908882){\color[rgb]{0,0,0}\makebox(0,0)[lt]{\begin{minipage}{0.3676456\unitlength}\raggedright $E_1$\end{minipage}}}%
    \put(-0.00631577,0.35525141){\color[rgb]{0,0,0}\makebox(0,0)[lt]{\begin{minipage}{0.3676456\unitlength}\raggedright $E_2$\end{minipage}}}%
    \put(0.51095406,0.10258511){\color[rgb]{0,0,0}\makebox(0,0)[lt]{\begin{minipage}{0.3676456\unitlength}\raggedright $E_3$\end{minipage}}}%
    \put(0.74059396,0.27729557){\color[rgb]{0,0,0}\makebox(0,0)[lt]{\lineheight{1.25}\smash{\begin{tabular}[t]{l}$E_4$\end{tabular}}}}%
  \end{picture}%
\endgroup%

    & ~ & \def\svgwidth{100pt}  
\begingroup%
  \makeatletter%
  \providecommand\color[2][]{%
    \errmessage{(Inkscape) Color is used for the text in Inkscape, but the package 'color.sty' is not loaded}%
    \renewcommand\color[2][]{}%
  }%
  \providecommand\transparent[1]{%
    \errmessage{(Inkscape) Transparency is used (non-zero) for the text in Inkscape, but the package 'transparent.sty' is not loaded}%
    \renewcommand\transparent[1]{}%
  }%
  \providecommand\rotatebox[2]{#2}%
  \newcommand*\fsize{\dimexpr\f@size pt\relax}%
  \newcommand*\lineheight[1]{\fontsize{\fsize}{#1\fsize}\selectfont}%
  \ifx\svgwidth\undefined%
    \setlength{\unitlength}{307.10810748bp}%
    \ifx\svgscale\undefined%
      \relax%
    \else%
      \setlength{\unitlength}{\unitlength * \real{\svgscale}}%
    \fi%
  \else%
    \setlength{\unitlength}{\svgwidth}%
  \fi%
  \global\let\svgwidth\undefined%
  \global\let\svgscale\undefined%
  \makeatother%
  \begin{picture}(1,0.72266633)%
    \lineheight{1}%
    \setlength\tabcolsep{0pt}%
    \put(0,0){\includegraphics[width=\unitlength,page=1]{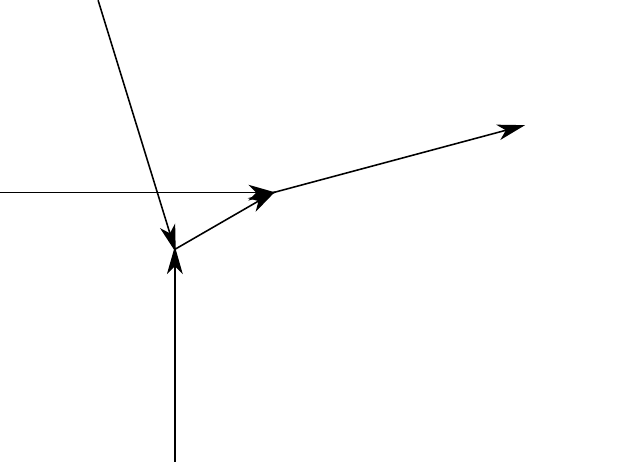}}%
    \put(0.01572765,0.72864919){\color[rgb]{0,0,0}\makebox(0,0)[lt]{\begin{minipage}{0.43047939\unitlength}\raggedright $E_1$\end{minipage}}}%
    \put(-0.00526083,0.40733525){\color[rgb]{0,0,0}\makebox(0,0)[lt]{\begin{minipage}{0.43047939\unitlength}\raggedright $E_2$\end{minipage}}}%
    \put(0.27875058,0.11776587){\color[rgb]{0,0,0}\makebox(0,0)[lt]{\begin{minipage}{0.43047939\unitlength}\raggedright $E_3$\end{minipage}}}%
    \put(0.69625924,0.4060662){\color[rgb]{0,0,0}\makebox(0,0)[lt]{\lineheight{1.25}\smash{\begin{tabular}[t]{l}$E_4$\end{tabular}}}}%
  \end{picture}%
\endgroup%
 \\ \\
    $[\zeta_{E_1},[\zeta_{E_2},\zeta_{E_3}]]$  &$=$& $[[\zeta_{E_1},\zeta_{E_2}],\zeta_{E_3}]$ &$+$& $[\zeta_{E_2},[\zeta_{E_1},\zeta_{E_3}]]$
    \end{tabular}
            \caption{In the space of translations of the incidence conditions there are codimension $1$ walls along which two $3$-valent vertices merge to one $4$-valent vertex. On one side of such a wall, this $4$-valent vertex can deform in one way, while on the other side it may deform in one or two ways.  The invariance of the tropical counts is then related to the Jacobi identity. The bracket here is the Schouten-Nijenhuis bracket (i.e., $l_2$). \label{JacobiFig}}
\end{figure}

In general, the signs in this Jacobi identity approach are surprisingly unwieldy.  However, given another Gerstenhaber algebra which $q$-deforms polyvector fields and which satisfies $l_2(\wt{\zeta}_{E_1},\wt{\zeta}_{E_2})=0$ whenever the $q\mapsto 1$ limit $[\zeta_{E_1},\zeta_{E_2}]$ is $0$, one can use this approach to prove refined invariance results.  Indeed, similar techniques were used in \cite[\S 3.3.3]{Man3} to prove a refined version of the Carl-Pumperla-Siebert \cite{CPS} Lemma on consistency of theta functions, interpreted in \cite{Man3} in terms of invariance of tropical counts. Invariance of the refined descendant tropical counts mentioned in Example \ref{ThetaEx} was obtained as a corollary, cf. \cite[Prop. 3.5]{Man3}.  We hope that Block-G\"ottsche invariants \cite{BG} (whose invariance was proved in \cite{IM} using the techniques of \cite{GM}) could be understood using this approach, along with the various other refined invariants defined in \cite{BlSh,GS,Man3,ShuRefined,SS,Blo}\footnote{We note that \cite{Blo} also introduces orientations on the tropical moduli space which could be used to address the unwieldy sign issue we mentioned above.  We also note that \cite[Thm. 6.8]{Blo} appears to rediscover our Theorem \ref{BracketMult}.} but at this point we do not know how to $q$-deform more than just the degree $0$ and $1$ parts of the polyvector field Gerstenhaber algebra (the $q$-deformation in these degrees essentially gives the quantum torus algebra and its adjoint action).

\subsection{Getzler's gravity algebra, string topology and symplectic cohomology}
In \cite{Getzler}, Section 4, Ezra Getzler introduced an algebraic structure induced on the equivariant cohomology of a topological conformal field theory by the functional integrals over the moduli space of smooth curves $\shM_{0,n}$ that he named a \emph{gravity algebra}, see also \cite{GK94} for the Koszul dual structure.
Given a graded BV-algebra $(A,\Delta)$, Getzler defines brackets 
$$\{a_1,...,a_k\}:=\Delta(a_1\cdot...\cdot a_k)-\sum_{i=1}^k (-1)^{\deg(a_1)+...+\deg(a_{i-1})} a_1\cdot...\cdot(\Delta a_i)\cdot...\cdot a_k.$$
After taking $(A=\bb{Z}[N]\otimes_\ZZ \Lambda^* M,\ell_1)$ as our BV-algebra (cf. \S \ref{l1BV}), a direct comparison proves the following observation: 
\begin{lem} 
Getzler's gravity operators on $A=\bb{Z}[N]\otimes_\ZZ \Lambda^* M$, when restricted to $A_0=\ker(\Delta)$, agree with the brackets $\ell_k$ defined in \eqref{eq-l-infty}.
\end{lem} 
The relationship of our $L_\infty$-structure with string topology is explained in detail in Appendix~\ref{app}.
We set $T=N_\RR/N \cong (S^1)^r$. If $\s LT$ denotes the free loop space of $T$, then by \eqref{iso-loop-module}, the ring $A$ is naturally isomorphic to
$$
A\cong H_{r-\bullet}(\s LT,\ZZ)
$$
when the latter is equiped with the loop product (Observation~\ref{Janko1}). 
By Abouzaid's theorem \cite[Corollary 6.1.2]{Abouzaid}, $A$ is now also identified with the symplectic cohomology of $T$ by a natural isomorphism
$$
A\cong \SH^\bullet(T^*T,\ZZ)
$$
of BV-algebras. Finally, the $L_\infty$-structure $\{l_k\}$ defined as a certain sign-twisting of the brackets $\{\ell_k\}$ from \eqref{eq-l-infty}  (cf. \S \ref{sec:lk}) is compatible with Chas-Sullivan's $L_\infty$-structure on equivariant cohomology of $\s LT$ by Observation~\ref{Janko2}.

\subsection{Counting special Lagrangian submanifolds}
The multiplicities of tropical curves that are computed by the algebraic structures that we introduce have recently been identified to agree with Joyce's weight of a Lagrangian submanifold \cite{MakRud,MikLag} whenever the tropical curve is used to produce a Lagrangian submanifold in a Lagrangian torus fibration. 
For the latter construction, see \cite{RS19,Matessi,Matessi2,MikLag,MakRud,Hicks,Hicks2}. The construction of torus fibrations and cycles of weight equal to the tropical multiplicity is vastly generalized topologically in \cite{RudZha1,RudZha2,RudZha3}.

\subsection{Relationship with previously known formulae in special cases}

Various already-known multiplicity formulae can be easily recovered from our algebraic structure theorems. In particular, Mikhalkin's formula \cite{Mi} for multiplicities of planar tropical curves is easily recovered from our Proposition \ref{SplitPoint}.  A formula for multiplicities of genus $0$ curves satisfying line conditions in three dimensions is given in \cite[Prop. 6.7]{MikLag}, and this can be recovered from our Theorem \ref{BracketMult} (in this setup, the wedge-products followed by contractions are interpreted as cross-products).

\subsection{Acknowledgements} 
We thank Lawrence Barrott, Mark Gross, Joachim Kock, Janko Latschev, Sven Meinhardt, Dan Pomerleano, Brent Pym, Nick Sheridan, Dmitry Tonkonog, and Yixian Wu for useful conversations.

\section{Review of tropical curves and their multiplicities}

\begin{ntn}
\label{TropNotation} For use throughout this paper, fix a lattice $N$ of finite rank $r\geq 0$, and let $M$ be the dual lattice $\Hom(N,\bb{Z})$.  For any lattice $L$, denote $L_{\bb{R}}:=L\otimes \bb{R}$.  Let $\langle\cdot,\cdot\rangle$ denote the pairing between a lattice and its dual.  We say $v\in L$ is primitive if it is not a positive multiple of any other element of $L$, and we say $v$ has index $k\ge 0$ in $L$ if $v=kv'$ for some primitive $v'\in L$, $k\in \bb{Z}_{\geq 0}$.  We denote the index of $v$ by $|v|$.  Given any subset $S\subset N_{\bb{R}}$, we let $\LL(S)$ denote the linear span of $S$ in $N_{\bb{R}}$, i.e., the $\bb{R}$-span of the set of vectors $u-v$ where $u,v\in S$.  We will denote $\LL_N(S):=\LL(S)\cap N$.  
\end{ntn}

\subsection{Tropical curves}
\label{BasicSection}
In this and the next subsection, we recall the basic definitions of tropical Gromov-Witten numbers, cf. \cite[\S 2]{MRud} for more details.

Let $\?{\Gamma}$ denote the topological realization of a finite connected graph.  Let $\Gamma$ be the complement of some subset of the $1$-valent vertices of $\?{\Gamma}$.  Let $\Gamma^{[0]}$, $\Gamma^{[1]}$, $\Gamma^{[1]}_{\infty}$, and $\Gamma^{[1]}_c$ denote the sets of vertices, edges, non-compact edges, and compact edges of $\Gamma$, respectively.  We equip $\Gamma$ with a ``weight-function'' $w:\Gamma^{[1]}\rar \bb{Z}_{\geq 0}$ and a ``genus-function'' $g:\Gamma^{[0]} \rar \bb{Z}_{\geq 0}$, subject to the requirement that univalent and bivalent vertices have positive genus. 

A {\bf marking} of $\Gamma$ is a bijection $\epsilon:I\rar \?{\Gamma}^{[0]}\setminus \Gamma^{[0]}$ for some index set $I$.  Let $E_i\in \Gamma^{[1]}_{\infty}$ denote the edge containing $\epsilon(i)$.  Let $I^{\circ}\subset I$ denote the set of $i\in I$ for which $w(E_i)=0$.
Denote by $(\Gamma,\epsilon)$ the data of $\Gamma$, the weight-function $w$, the genus-function $g$, and the marking.

\begin{dfn}\label{TropCurveDfn}
A {\bf parameterized tropical curve} $(\Gamma,\epsilon,h)$ is data $(\Gamma,\epsilon)$ as above, along with a continuous map $h:\Gamma\rar N_{\bb{R}}$ such that 
\begin{enumerate}
\item For each edge $E\in \Gamma^{[1]}$ with $w(E)>0$, $h|_E$ is a proper embedding into an affine line with rational slope. For $E\in \Gamma^{[1]}$ with $w(E)=0$, $h(E)$ is a point.
\item  For every $V\in \Gamma^{[0]}$, the following {\bf balancing condition} holds.  For each edge $E\ni V$, denote by $u_{(V,E)}$ the primitive integral vector emanating from $h(V)$ into $h(E)$ (or $u_{(V,E)}\coloneqq 0$ if $h(E)$ is a point). Then
\begin{align*}
\sum_{E\ni V} w(E) u_{(V,E)} =0.
\end{align*}
\end{enumerate}
Furthermore, for each contracted compact edge $E$, we have the additional data of a ``length'' in $\bb{R}_{>0}$.

For unbounded edges $E_i \ni V$, we may denote $u_{(V,E_i)}$ simply as $u_{E_i}$ or $u_i$.  
Similarly, for any edge $E$, we may simply write $u_E$ when the vertex is either clear from context or unimportant (e.g., as in $\bb{Z}u_{E}$).  
For each edge, we arbitrarily fix a labelling of its vertices as $\partial^+ E$ and $\partial^- E$, possibly writing just $\partial E$ if $E$ contains only one vertex.  If $w(E)=0$ and $V\in E$, we take $u_E\coloneqq u_{(V,E)} = 0$.

An {\bf isomorphism} of parameterized tropical curves $(\Gamma,\epsilon,h)$ and $(\Gamma',\epsilon',h')$ is a homeo\-mor\-phism $\Phi:\Gamma\rar \Gamma'$ respecting the weights, genera, and markings such that $h=h'\circ \Phi$.  A {\bf tropical curve} is then defined to be an isomorphism class of parameterized tropical curves.  We will use $(\Gamma,\epsilon,h)$ to denote the isomorphism class it represents and will often abbreviate this as simply $h$ or $\Gamma$.
\end{dfn}

\begin{rmk}\label{noV}
If $\Gamma$ is nonempty but contains no vertices, then $\?{\Gamma}$ consists of two univalent vertices connected by an edge.  Then $I$ labels these univalent vertices, hence $I$ labels the two unbounded directions of $\Gamma$, which we view as the flags of $\Gamma$.  With this convention, the notions of type and degree are easily extended to curves $\Gamma$ with no vertices, but to simplify the exposition, we assume for the rest of this section that  $\Gamma^{[0]}\neq \emptyset$.  See \cite[Rmk. 4.17]{MRud} for some details on this case.
\end{rmk}

If $b_1(\Gamma)$ denotes the first Betti number of $\Gamma$, the {\bf genus} of a tropical curve $\Gamma$ is defined as
$$
g(\Gamma)= b_1(\Gamma)+\sum_{V\in\Gamma^{[0]}}g(V).
$$

Let $\Flags(\Gamma)$ denote the set of flags of $\Gamma$, i.e., pairs $(V,E)$ with $E\in \Gamma^{[1]}$ and $V$ a vertex of $E$.  The {\bf type} $\f{u}$ of a marked tropical curve is the data of the underlying graph $\Gamma$, $w$, $g$, $\epsilon$, plus the data of the map $u:\Flags(\Gamma)\rar N$, $(V,E)\mapsto w(E)u_{(V,E)}$.

Given a tropical curve,  the {\bf degree} $(I,\Delta)$, or $\Delta$ for short, is the data of the index set $I$ from the marking, along with the corresponding map $\Delta:I\rar N$, $\Delta(i)=w(E_i)u_i$.

Let $\val(V)$ denote the valence of a vertex $V$.  Define the over-valence $\ov(V)\coloneqq \val(V)+3g(V)-3$, and \[\ov(\Gamma)\coloneqq \sum_{V\in \Gamma^{[0]}} \ov(V).\]

The moduli space $\f{T}_{g,\Delta}$ of marked tropical curves of genus $g$ and degree $\Delta$ is a polyhedral complex whose faces correspond to tropical curve types.  If $\Gamma\in \f{T}_{g,\Delta}$ has type $\f{u}$, then the {\bf expected dimension} for the face
 $F_{\f{u}}$ corresponding to $\f{u}$ is
\begin{align*}
    d^{\trop}_{g,\f{u}}\coloneqq  \#I+(r-3)(1-g)-\ov(\Gamma).
\end{align*}
We say that tropical curves of type $\f{u}$ are {\bf non-superabundant} if they contain no contracted loops or higher-genus vertices and the actual dimension of $F_{\f{u}}$ equals this expected dimension.

\subsection{Tropical Gromov-Witten numbers}

\begin{dfn}\label{Constraints}
An {\bf affine constraint} $\A$ is a tuple $(A_i)_{i\in I}$ of affine subspaces of $N_{\bb{R}}$.  A marked tropical curve $(\Gamma,\epsilon,h)$ {\bf matches the constraint $\A$} if $h(E_i)\subset A_i$ for all $i\in I$.

Recall that $I^{\circ}\coloneqq \{i\in I|w(e_i)=0\}$.  Consider a tuple $\Psi\coloneqq(s_i)_{i\in I^{\circ}}\in \bb{Z}_{\geq 0}^{{m}}$.  For each $V\in \Gamma^{[0]}$, denote $I_V^{\circ}\coloneqq\{i\in I^{\circ}:E_i\ni V\}$.
 We say $(\Gamma,\epsilon,h)$ satisfies $\Psi$ if for each vertex $V$ we have 
 \begin{align}\label{PsiTropDfn}
    \ov(V) \geq \sum_{i\in I_V^{\circ}} s_i.
\end{align}

We are interested in the space
\[\f{T}_{g,\Delta}(\A,\Psi)\subset \f{T}_{g,\Delta}\]
of marked tropical curves of genus $g$, degree $\Delta$, matching the constraints $\A$ and satisfying the $\psi$-class conditions $\Psi$.  We write $\f{T}_{g,\f{u}}(\A,\Psi)$ for the subspace corresponding to tropical curves of type $\f{u}$.

For a marked vertex $V \in \Gamma^{[0]}$, when \eqref{PsiTropDfn} is an equality, let $\langle V \rangle$ denote the multinomial coefficient
\begin{align}\label{V}
\langle V \rangle \coloneqq\binom{{\ov(V)}}{{s_{i_1},\ldots,s_{i_{m_V}}}}_{i_j\in I_V^{\circ}} \coloneqq \frac{\ov(V)!}{\prod_{i\in \mu^{-1}(V)} s_i!}.
\end{align}
If no contracted edges contain $V$, then $\langle V\rangle\coloneqq1$.
\end{dfn}

When we say $\A$ is {\bf generic}, we mean that the spaces $A_i$ are generic translates of their corresponding linear spans $\LL(A_i)$ (cf. Notation \ref{TropNotation}).  For an edge $E\in \Gamma^{[1]}$, we will write $\LL(E)$ and $\LL_N(E)$ to mean $\LL(h(E))$ and $\LL_N(h(E))$, respectively.  That is, $\LL(E)=\{0\}$ if $w(E)=0$, and $\LL(E)=\bb{R}u_E$ otherwise.  

\begin{lemdfn}\label{MultDfn}
Let $\Gamma$ be a non-superabundant tropical curve of type $\f{u}$ in $\f{T}_{g,\Delta}(\A,\Psi)$ for a generic choice of $\A$ (generic in the space of translations of the incidence conditions).  Suppose that 
\begin{align}\label{GeneralDimension}
\sum_{i\in I} \codim(A_i) = d^{\trop}_{g,\f{u}}.
\end{align}
In this case, $\Gamma$ is an isolated point of $\f{T}_{g,\Delta}(\A,\Psi)$, and we say that $\Gamma$ is {\bf rigid} (with respect to $\A$ and $\Psi$).  We call $\f{T}_{g,\Delta}(\A,\Psi)$ rigid if every $\Gamma \in \f{T}_{g,\Delta}(\A,\Psi)$ is rigid, and in this case, $\f{T}_{g,\Delta}(\A,\Psi)$ is finite.

For any $(\Gamma,h)\in \f{T}_{g,\Delta}(\A,\Psi)$, we have a map
\begin{align}\label{D}
\Phi\coloneqq\prod_{V\in \Gamma^{[0]}}N &\rar \left(\prod_{E\in \Gamma^{[1]}_c} N/\LL_N(E)\right) \times \left(\prod_{i\in I} N/ \LL_N(A_i)\right) \\
H &\mapsto ((H_{\partial^+E}-H_{\partial^-E})_{E\in \Gamma^{[1]}_c},(H_{\partial E_j})_{j\in I}). \notag
\end{align}
Let $\Phi_{\bb{R}}=\Phi \otimes \bb{R}$, so $\ker \Phi_{\bb{R}}$ is naturally identified with the tangent space to $\f{T}_{g,\f{u}}(\A,\Psi)$ at $\Gamma$ as in \cite[Prop. 2.10]{MRud}.

In particular, when $\Gamma$ is a rigid tropical curve, $\Phi$ is a finite-index inclusion of lattices.  We denote
\begin{align}\label{DGamma}
    \f{D}_{\Gamma}\coloneqq \inde(\Phi)
\end{align} and
\begin{align}\label{MultGamma}
\Mult(\Gamma)\coloneqq\f{D}_{\Gamma} \prod_{E\in \Gamma^{[1]}_c} w(E).
\end{align}
If there is ambiguity about which conditions $\A$ are being imposed, we will write $\Mult_{\A}(\Gamma)$.

If $\f{T}_{g,\Delta}(\A,\Psi)$ is rigid, we define the \textbf{tropical descendant Gromov-Witten numbers} as follows:
\begin{align}\label{GWtrop}
\GW_{g,\Delta}^{\trop}(\A,\Psi)\coloneqq\sum_{(\Gamma,\mu,\epsilon,h)\in \f{T}_{g,\Delta}(\A,\Psi)} \frac{\Mult(\Gamma)}{|\Aut \Gamma|}\prod_{V\in \Gamma^{[0]}} \langle V \rangle.
\end{align}
\end{lemdfn}

It was proved in \cite[Thm 1.1]{MRud} that this quantity coincides with the corresponding descendant log Gromov-Witten invariant (as well as a naive algebraic count) for projective toric varieties with cocharacter lattice $N$, defined over an algebraically closed characteristic $0$ field $\kk$.

We note that special cases of the correspondence result \cite[Thm 1.1]{MRud} were previously proved in many other works.  In particular, \cite{MRud} built on techniques of \cite{NS}, which considered genus $0$ cases without $\psi$-class conditions.  Arbitrary genus cases in dimension $2$ with only point conditions were dealt with in \cite{Mi}, where multiplicities were already defined as products of vertex multiplicities.  As discussed below, \cite{Ran,AGr} use tropical intersection theory techniques to prove the correspondence theorems in genus $0$ (with and without $\psi$-classes, respectively).  We also note that the first correspondence theorems involving $\psi$-classes were proved in special two-dimensional genus $0$ cases by \cite{MR} and \cite{GrP2}.

\subsection{The tropical intersection-theoretic description of multiplicities}

In genus $0$, the above-mentioned correspondence between tropical and algebraic Gromov-Witten counts was proved in \cite{Ran,AGr} in terms of tropical intersection theory, a quite different approach from that of \cite{NS,MRud}.  In particular, this indicates that the multiplicity $\Mult(\Gamma)$ of \eqref{MultGamma} can be expressed in terms of the tropical intersection theory developed in \cite{AR,Rau}.  For our goal of understanding our tropical curve multiplicities, it will suffice for us to understand tropical intersections of weighted rational subspaces of some $L_{\bb{R}}$, i.e. rational-slope linear subspaces with an associated weight in $\bb{Z}_{\geq 1}$.  We call a weighted rational subspace primitive if this weight is $1$.  Since the tropical intersection product is linear, it suffices to understand the primitive cases.  Intersections of weighted rational subspaces are then characterized by the following lemma, and for our purposes the reader may take this as the definition of the tropical intersection product.  Here and below, we use $\prod$ to denote the tropical intersection product of a collection of tropical cycles, and $\bigcap$ to denote the set-theoretic intersection.
\begin{lem}\label{TropIntIndex}
Given a finite-rank lattice $L$, let $\{A_i\}_{i\in I}$ be a collection of primitive weighted rational subspaces of $L\otimes \bb{R}$.  Then for the intersection product of these classes we have $\prod_{i} A_i = w \bigcap_i A_i$, where $w$ is the index of the map $L\rar \bigoplus_i L/(L\cap A_i)$ (and $w=0$ if this map is not finite-index).
\end{lem}
\begin{proof}
The case where each $A_i$ is a hyperplane is just \cite[Lem. 1.4]{Rau} (for each $h_i$ there having primitive slope).  The general case follows after noting that any $A_i$ can be realized as a tropical product of hyperplanes.
\end{proof}

For each edge $E\in \Gamma^{[1]}$, there is a corresponding factor $N/\LL_N(E)$ or $N/\LL_N(A_i)$ 
in the codomain of $\Phi$, and we define $\Phi_E$ to be the composition of $\Phi$ with the projection onto this factor.  Then $\ker(\Phi_E\otimes \bb{R})$ defines a linear subspace, hence a tropical cycle in $\domain(\Phi_{\bb{R}})$.  If $E$ is compact, this tropical cycle is a diagonal class between the two copies of $N_{\bb{R}}$ corresponding to the vertices of $E$, and we denote the class by $[\Delta_E]$.  If $E=E_i$ is not compact, we denote the corresponding class by $[A_i]$ since it is the class of the pullback of $A_i$ by the evaluation map corresponding to $E_i$.

\begin{prop}\label{MultTropIntersection}
\begin{align}\label{MultTropIntersectionFormula}
    \Mult(\Gamma)=\int_{\domain(\Phi_{\bb{R}})} \left(\prod_{E\in \Gamma^{[1]}_c} w(E)[\Delta_E]\right) . \left( \prod_{i\in I} [A_i] \right).
\end{align}
\end{prop}

\begin{proof}
It follows easily from Lemma \ref{TropIntIndex} and the definition of $\Phi$ that $\f{D}_{\Gamma}$ is given by the right-hand side of \eqref{MultTropIntersectionFormula} without the $w(E)$-factors.  The claim then follows immediately after multiplying by these weights.  
\end{proof}

\begin{rmk}
For readers familiar with our paper \cite{MRud}, we note here that we could have directly used \eqref{MultTropIntersectionFormula} (times $\prod_V \langle V \rangle$) as our definition of multiplicity in the proof of the correspondence theorem \cite[Thm 1.1]{MRud}.  Indeed, these tropical intersections have a geometric interpretation directly applicable to our proof there as follows:  let
\begin{align*}
    \Phi^{\circ}\coloneqq  \prod_{V\in \Gamma^{[0]}}N &\rar \prod_{E\in \Gamma^{[1]}_c} N/w(E)\LL_N(E) \\
H &\mapsto (H_{\partial^+E}-H_{\partial^-E})_{E\in \Gamma^{[1]}_c}.
\end{align*}
The combination of \cite[Prop. 4.10 and Lem. 4.11]{MRud} says that the space of log curves in $\s{X}_0^{\dagger}$ (cf. \textit{loc. cit.} for the notation) with tropicalization $\Gamma$ is a $\ker(\Phi^{\circ}_{\kk^*})$-torsor over $\s{M}(\Gamma)\coloneqq (\prod_V \s{M}_{g(V),\val(V)})/\Aut(\Gamma)$.   The $\psi$-classes exactly serve to cut out $\prod_V \langle V\rangle$ points in the base $\s{M}(\Gamma)$, each with multiplicity $\frac{1}{|\Aut \Gamma|}$, so we can focus on $\ker(\Phi^{\circ}_{\kk^*})$.  We want to show that the number of points in the intersection of the algebraic cycles $[Z_{A_i}]$ and $[Z_{A_i,u_i}]$ (cf. \cite[\S 3.2.1]{MRud}) in $\ker(\Phi^{\circ}_{\kk^*})$ is given by \eqref{MultTropIntersectionFormula}.  Indeed, after observing that $\ker(\Phi^{\circ}_{\kk^*})$ is the intersection in $\domain(\Phi_{\kk^*})$ of the diagonal classes $w(E)[\Delta_E^{\kk^*}]$ corresponding to the compact edges, it is clear that the tropical intersection of \eqref{MultTropIntersectionFormula} is exactly the tropicalization of the intersection of toric cycles appearing in the algebraic setup.  
\end{rmk}

\subsection{Tropical intersections, wedge products, and a Frobenius algebra}
The following reinterpretations of Lemma \ref{TropIntIndex} will be instrumental in \S \ref{SectionTrQFT}.

Let $L$ be an arbitrary lattice.  Given a linear rational-slope subspace $A\subset L_{\bb{R}}$ of weight $w$, let $\alpha_A$ denote the unique-up-to-sign element \begin{align}\label{alphaA}
\alpha_A\in \Lambda^{\codim A} L^*
\end{align}
of index $w$ whose restriction to $A$ is trivial.

\begin{lem}\label{Trop-Wedge}
For $\{A_i\}_i$ a collection of weighted rational subspaces of $L_{\bb{R}}$, and for $A\coloneqq \prod_i A_i$, we have $\alpha_A = \pm \bigwedge_i \alpha_{A_i}$.
\end{lem}

The sign ambiguity in the Lemma is inconvenient.  It can be avoided using the following ``squaring'' trick which will also prevent more serious sign issues later on.  For any lattice $L$ and elements $a_1,\ldots,a_k\in L$, if $\alpha=a_1\wedge \cdots \wedge a_k$, we denote 
\begin{align}\label{squaring}
    \alpha^{\Box} \coloneqq  \bigwedge_{i=1}^k ((a_i,0)\wedge (0,a_i)) \in \Lambda^{\even} (L\oplus L).
\end{align}
Equivalently, $\alpha^{\Box}=(-1)^{\deg(\alpha)(\deg(\alpha)-1)/2}(\alpha,0)\wedge (0,\alpha)$.  If $\alpha\in \Lambda^0 L = \bb{Z}$, then $\alpha^{\Box}\coloneqq \alpha^2$.  We obtain a canonical element
\begin{align}\label{ThetaL}
    \Theta_L^{\Box}\in \Lambda^{\topp}(L\oplus L),
\end{align}
where $\Theta_L$ is either choice of primitive element of $\Lambda^{\topp} L$.  The exterior algebra 
\begin{align}\label{CL}
    \s{C}_L\coloneqq \Lambda^*(L\oplus L)
\end{align} then becomes a graded-commutative Frobenius algebra\footnote{Recall that a Frobenius algebra over $R$ is an associative $R$-algebra $A$ together with an $R$-linear trace map $\Tr:A\rar R$ such that the pairing $\Tr:A\otimes_R A\rar R$, $\Tr(a \otimes b)\coloneqq \Tr(ab)$ is non-degenerate.  By graded-commutative, we mean that the multiplication is graded-commutative, and the Frobenius trace preserves the parity of the grading.  This parity-preservation is necessary for associating a closed-string $2D$ TQFT to the Frobenius algebra.  For odd-dimensional $L$, the Frobenius algebra structure on $\Lambda^* L$ would not be graded-commutative.  This motivates the squaring trick.} over $\bb{Z}$ with trace 
\begin{align*}
    \Tr_L:\s{C}_L\rar \bb{Z}
\end{align*} given by projecting onto $\Lambda^{\topp} (L\oplus L)$ and then composing with the unique map $\Lambda^{\topp} (L\oplus L)\rar \bb{Z}$ taking $\Theta_L^{\Box}$ to $1$.

Lemma \ref{Trop-Wedge} immediately implies the following:
\begin{lem}\label{Trop2Wedge}
Notation as in Lemma \ref{Trop-Wedge}.   Then
\begin{align*}
    \int_{L_{\bb{R}}} \prod_i A_i = \sqrt{\Tr_L\left(\bigwedge_i \alpha_{A_i}^{\Box} \right)}.
\end{align*}
\end{lem}

\section{Tropical quantum field theory}\label{SectionTrQFT}

\subsection{The definition of 2-dimensional tropical quantum field theory}\label{TrQFTdef}
In this section we define the notion of a 2D tropical quantum field theory\footnote{We view the TrQFT's introduced here as being two-dimensional, even though the tropical curves are one-dimensional, because the TQFT's it most closely resembles are two-dimensional.  Indeed, the log curves associated to the tropical curves have real dimension two, and we suspect our 2D TrQFT's can thus be viewed as the tropicalization of a logarithmic version of a 2D TQFT, cf. Remark \ref{logTQFT}.} (TrQFT for short) with target space $N_{\bb{R}}$.  We view this as a tropical analog of a 2D topological quantum field theory (TQFT).  

We begin by defining a symmetric monoidal category $\Trop$ which will for us play the role that $2\Cob$ (the category whose objects are disjoint unions of circles and whose morphisms are $2$-dimensional cobordisms) typically plays for a 2D TQFT.  An object of $\Trop$ is a \textbf{tropical degree}, by which we mean the data of a finite index-set $I$ along with a map $\Delta:I\rar N$.  Note that no balancing condition is imposed.  Here, two tropical degrees $(I_1,\Delta_1)$ and $(I_2,\Delta_2)$ are identified as the same object if there exists a bijection $i:I_1\risom I_2$ such that $\Delta_1=\Delta_2\circ i$.  We have an monoidal operation $\sqcup$ which, given two such objects $(I_1,\Delta_1)$ and $(I_2,\Delta_2)$, produces a third object $(I_1\sqcup I_2,\Delta_1\sqcup \Delta_2)$, where $I_1\sqcup I_2$ is the disjoint union of $I_1$ and $I_2$, and $\Delta_1\sqcup \Delta_2$ is the map taking $j$ to $\Delta_1(j)$ if $j\in I_1$ and $\Delta_2(j)$ if $j\in I_2$.

We note that the empty tropical degree $\emptyset\to N$ is the identity element for $\sqcup$.  For convenience, we will often write $[n]$ to denote the object \begin{align}\label{n}
    \Delta:\{1\}\rar N
\end{align} with $\Delta(1)=n$.

\begin{dfn}\label{TropCobDfn}
A \textbf{tropical cobordism} of tropical degree $\Delta:I\rar N$ is an equivalence class represented by the following data:
\begin{itemize}
    \item A finite graph $\?{\Gamma}$ {(not necessarily connected)};
    \item A ``marking'' $\epsilon:I\hookrightarrow \?{\Gamma}^{[0]}$, with image in the set of $1$-valent vertices of $\?{\Gamma}^{[0]}$.  Let $\Gamma\coloneqq \?{\Gamma}\setminus \epsilon(I)$.  For each $i\in I$, let $E_i$ denote the edge containing $\epsilon(i)$;
    \item A function $u:\Flags(\?{\Gamma})\rar N$.  If $V_1,V_2$ are the vertices of $E$, we require $u(V_1,E)=-u(V_2,E)$.  We also require $u(\epsilon(i),E_i)=-\Delta(i)$ for each $i\in I$;  
    \item A ``genus-function'' $g:\Gamma^{[0]}\rar \bb{Z}_{\geq 0}$.
\end{itemize}
The equivalence relation on the set of data is generated by
\begin{itemize}
    \item Isomorphisms of $\?{\Gamma}$ which respect $u$, $\epsilon$, and $g$, and
    \item  The operation that takes $E\in \Gamma^{[1]}_c\cap u^{-1}(0)$ with vertices $V$ and $V'$ and then and contracts $E$, identifying $V$ with $V'$.  If $E$ is self-adjacent, i.e., if $V=V'$, then we increase $g(V)$ by $1$ when we contract $E$.  Otherwise, the value of $g$ of the resulting vertex is the sum $g(V)+g(V')$.
\end{itemize}

In other words, a tropical cobordism is a possibly unbalanced tropical curve, up to type, and modulo contractions of compact weight-zero edges as in the equivalence relation above.
\end{dfn}

We can now define the morphisms of $\Trop$.  Given objects $(I_1,\Delta_1)$ and $(I_2,\Delta_2)$ as above (sometimes abbreviated as just $\Delta_1$ and $\Delta_2)$, $\Hom(\Delta_1,\Delta_2)$ is defined as the set of equivalence classes of tropical cobordisms of tropical degree $$(I_1\sqcup I_2, \Delta_1\sqcup (-\Delta_2)).$$  Note here that we negate the target degree.

Next suppose we have tropical cobordisms $\Gamma_{12}\in \Hom(\Delta_1,\Delta_2)$ and $\Gamma_{34}\in \Hom(\Delta_3,\Delta_4)$, and consider a set $J$, identified with a subset of $I_2$ and with a subset of $I_3$, such that $\Delta_{2}(j)=\Delta_{3}(j)$ for each $j\in J$. 
 Then we have a composition $\circ_{J}$ obtained by gluing $\Gamma_{12}$ and $\Gamma_{34}$ along the edges $E_{12,j}$ and $E_{34,j}$ associated to $j$ for each $j\in J$.  By ``gluing,'' we mean that we form a new tropical cobordism by removing $E_{12,j}$ and $E_{34,j}$ from $\Gamma_{12}$ and $\Gamma_{34}$, respectively, and then replacing these by a new compact edge $E_{j}$ between the vertices $\partial E_{12,j}$ and $\partial E_{34,j}$.  We then set 
 \begin{align*}
     u(\partial E_{12,j},E_j)=u(\partial E_{12,j},E_{12,j}) \qquad \mbox{and} \qquad u(\partial E_{34,j},E_j)=u(\partial E_{34,j},E_{34,j}).
 \end{align*}  Note that the assumption $\Delta_2(j)=\Delta_3(j)$ ensures that the condition $u(\partial E_{34,j},E_j)=-u(\partial E_{12,j},E_j)$ is satisfied.  The remaining data of the new curve is inherited in the obvious way.  In particular, when $J=I_2=I_3$, this $\circ_J$ gives the composition law for the category.
 
Note that the tropical cobordism with no vertices and with a single edge of weighted directions $\pm n\in N$ gives the identity morphism for the object $[n]$.  We have thus constructed our symmetric monoidal category $\Trop$.

\begin{dfn}
A two-dimensional \textbf{tropical quantum field theory} (TrQFT) is a functor $F$ of symmetric monoidal categories from $\Trop$ to another symmetric monoidal category $\f{C}$. 
\end{dfn}

We will always denote the monoidal operation on the target category $\f{C}$ by $\otimes$.

\begin{figure}[htb]
    \centering
    \includegraphics[width=.96\textwidth]{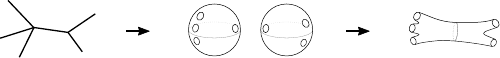}
    \caption{The map $\forget:\Trop\rar 2\Cob$ in an example.
    \label{fig:TrQFT-2-TQFT}}
\end{figure}
\begin{eg}\label{forgetEx}
There is a symmetric monoidal functor $\forget:\Trop\rar 2\Cob$ associating a circle to each element of $I$ and a $2$-cobordism to each tropical cobordism $\Gamma$. 
The way it works is sketched in Figure~\ref{fig:TrQFT-2-TQFT}.
To obtain $\forget(\Gamma)$ here, we first view $\Gamma$ as the dual graph to a pre-stable marked curve.  The cobordism is then constructed by treating markings as punctures, and treating nodes as pairs of punctures glued together. This yields a TrQFT, and furthermore, any TQFT can be pulled back via $\forget$ to yield a TrQFT.  In the reverse direction, we have a section of $\forget$ which naturally identifies $2\Cob$ with the subcategory of $\Trop$ whose objects are those of the form $\Delta:I \rar \{0\}\subset N$, and whose morphisms consist of tropical cobordisms whose flags all have direction $u=0$.  Thus, every TrQFT includes the data of a TQFT via restriction to this subcategory.
\end{eg}

\begin{rmk}
An earlier version of this paper took objects of $\Trop$ to be the projections of our tropical degrees to $\?{N}\coloneqq N/\{\pm \Id\}$, with similar modifications to morphisms and compositions.  While sufficiently general for our purposes here, that definition of $\Trop$ is less general (and less natural) than what we have presented here --- a TrQFT $F:\Trop\rar \f{C}$ in the present version corresponds to a TrQFT in the older framework if and only if $F([n])=F([-n])$ for each $n\in N$.
\end{rmk}

In the case of interest to us, the target symmetric monoidal category will be the category of super $\bb{Z}$-modules, which we denote by $$s\bb{Z}\Mod.$$  I.e., $s\bb{Z}\Mod$ is the category of $\bb{Z}/2\bb{Z}$-graded Abelian groups, with tensor product as the monoidal operator, and with braiding $\tau_{V,W}:V\otimes W\rar W\otimes V$ taking $v\otimes w$ to $(-1)^{\deg(v)\deg(w)} w\otimes v$, where $\deg(v)$ and $\deg(w)$ are the degrees of the homogeneous elements $v\in V$ and $w\in W$.

\subsection{Tropical flows and an algebraic characterization of TrQFT's}\label{TrAlg}

It is a standard fact (proved in \cite{Abrams}) that the data of a 2D TQFT valued in the category of vector spaces is equivalent to the data of a commutative Frobenius algebra.  More generally, a TQFT is a commutative Frobenius object in whatever the target symmetric monoidal category is.  In particular, when the target is the category of super $R$-modules for some ring $R$, a commutative Frobenius object is the same as a supercommutative Frobenius $R$-algebra.  See \cite{KockTQFT} (particularly \S 3.3.3) for a nice explanation of this generality.

The unit/counit and product/coproduct for the Frobenius object correspond to cups/caps and pairs of pants, respectively, in the category $2\Cob$.  Thus, the image of any 2-cobordism $C$ under the TQFT functor can be understood by taking a handle-body decomposition of $C$.  Gluing components of this decomposition along common boundary curves corresponds to composing the corresponding morphisms.
We will extend this to give a similar characterization of a TrQFT. First, we need a new definition and some notation:

 \begin{dfn}\label{flowdef}
Let $\Gamma$ represent a tropical cobordism.  Let $\Gamma'$ denote the graph obtained from $\Gamma$ by inserting a bivalent vertex in the middle of each edge, then adding univalent vertices to compactify the non-compact edges (i.e., reinserting the vertices of $\?{\Gamma}\setminus \Gamma$).  If $E\in \Gamma$ has no vertices, then we also insert a vertex in the middle of $E$ in addition to the two univalent vertices at the ends of $E$.  A \textbf{tropical flow} on $\Gamma$ is a choice of acyclic quiver structure on $\Gamma'$. If $\Gamma\in \Hom(\Delta_1,\Delta_2)$, then we require the vertices in $\epsilon(I_1)$ to be sources, and we require the vertices in $\epsilon(I_2)$ to be sinks.
 \end{dfn}
 
 In the proof of Theorem \ref{TrQFT-Equivalent-Data} below, it will be useful to have a pictorial representation for genus $0$ tropical cobordisms which have at most one non-univalent vertex and at most two edges.  When drawing such an element of $\Hom(\Delta_1,\Delta_2)$, we will do the following:
 \begin{itemize}
     \item We write $\bullet$ to indicate a genus $0$ vertex of $\Gamma$.
     \item We use arrow tails to indicate vertices in $\epsilon(I_1)$ and arrow heads to indicate vertices in $\epsilon(I_2)$.
     \item We label each arrow tail with the corresponding $\Delta_1(i)$ for $i\in I_1$, and we label each arrow head with the corresponding $\Delta_2(i)$ for $i\in I_2$.  Note that the vector from the label always points in opposite direction of the arrow head/tail,
 \end{itemize}
 For example, $\stackrel{-n\qquad n}{\lrsource}$ represents a tropical cobordism in $\Hom([-n]\sqcup[n],\emptyset)$, while $\stackrel{0\qquad}{\leftsource}\bullet\stackrel{\quad n}{\rightsink}$ represents an element of $\Hom([0],[n])$. Note that $\stackrel{n\qquad n}{\lsourcersink}$ is the identity morphism in $\Hom([n],[n])$.  Compositions of tropical cobordisms glue arrow heads to arrow tails of the same label, e.g.,
 \begin{align}\label{1L}
     \stackrel{-n\qquad n}{\lrsource} \circ ( \stackrel{0\qquad}{\leftsource}\bullet\stackrel{\qquad -n}{\rightsink} \sqcup     \stackrel{n\qquad n}{\lsinkrsource}) = \stackrel{0\qquad}{\leftsource} \bullet \stackrel{\qquad n}{\rightsource}
 \end{align}
 
 We shall denote the monoidal identity in $\f{C}$ by $1_{\f{C}}$, and we denote the braiding morphisms $V\otimes W \rar W\otimes V$ in $\f{C}$ by $\tau_{V,W}$.
 
\begin{thm}\label{TrQFT-Equivalent-Data}
The following data is equivalent to the data of a TrQFT $F:\Trop \rar \f{C}$:
\begin{enumerate}
\item A commutative Frobenius object $\s{C}_0$ in $\f{C}$.  We denote the Frobenius trace of $\s{C}_0$ by $\Tr_0$;
\item For each $n\in N\setminus \{0\}$, an object $\s{C}_n$ of $\f{C}$.  Let $\id_n$ denote the identity morphism on $\s{C}_n$;
\item For each $n\in N\setminus \{0\}$, a pair of morphisms 
\begin{align*}
    \Tr_n:\s{C}_n \otimes \s{C}_{-n} \rar 1_{\f{C}} \qquad \mbox{and} \qquad \Tr_n^{\vee}:1_{\f{C}}\rar \s{C}_n\otimes \s{C}_{-n}.
\end{align*}
For each $n\in N\setminus\{0\}$, these morphisms must satisfy the following:
\begin{enumerate}
    \item The snake relations:
    \begin{align*}
        (\Tr_n \otimes \id_n)\circ \Tr_{-n}^{\vee} = \id_n = (\id_n \otimes \Tr_{-n})\circ \Tr_n^{\vee};
    \end{align*}
    \item Compatibility with the braiding:
\begin{align}\label{Tr-tau}
    \Tr_n \circ  \tau_{\s{C}_{-n},\s{C}_{n}}= \Tr_{-n} \qquad \mbox{and} \qquad \tau_{\s{C}_{-n},\s{C}_{n}}\circ \Tr^{\vee}_{-n}= \Tr^{\vee}_{n};
\end{align}
\end{enumerate}
\item Morphisms $\kappa_n:\s{C}_n \rar \s{C}_0$ and $\kappa_n^{\vee}:\s{C}_0 \rar \s{C}_n$ for each $n\in \?{N}\setminus \{0\}$ such that 
\begin{align}\label{TrKappa}
   \Tr_{-n}\circ (\kappa_{-n}^{\vee}\otimes \Id_{n}) &=  \Tr_0 \circ (\Id_0 \otimes \kappa_n)\qquad \hbox{ as morphisms }\s{C}_0\otimes \s{C}_n\ra 1_{\f{C}}, \hbox{and dually,}\\
    (\kappa_{-n}\otimes \Id_{n})\circ  \Tr^{\vee}_{-n} &= (\Id_0\otimes \kappa_n^{\vee})\circ \Tr^{\vee}_0\qquad\hbox{ as morphisms } 1_{\f{C}}\ra \s{C}_0\otimes \s{C}_n. 
    \label{TrKappaDual}
\end{align}
\end{enumerate}
\end{thm}

\begin{proof}
For each $n\in N$, we have an object $[n]\in \Trop$ as in \eqref{n}, and then the object $\s{C}_n\in \f{C}$ is $F([n])$.  The fact that $F([0])$ must be a commutative Frobenius object follows from Example \ref{forgetEx}.

The morphisms from the statement of the theorem are then obtained from the TrQFT as follows:
\begin{align*}
    \Tr_n &\coloneqq  F(\stackrel{n\qquad -n}{\lrsource}),\\
    \Tr_n^{\vee} &\coloneqq  F(\stackrel{n\qquad -n}{\lrsink}), \\
    \kappa_n&\coloneqq  F( \stackrel{0\qquad}{\leftsink}\bullet\stackrel{\qquad n}{\rightsource}),\\
    \kappa_n^{\vee}&\coloneqq F( \stackrel{0\qquad}{\leftsource}\bullet\stackrel{\qquad n}{\rightsink}).
\end{align*}
With these definitions, the composition on the left side of \eqref{TrKappa} is $F$ applied to the composition of \eqref{1L}.  Similar computations show that the right-hand side of \eqref{TrKappa} is also equal to $F(\stackrel{0\qquad}{\leftsource} \bullet \stackrel{\qquad n}{\rightsource})$, while both sides of \eqref{TrKappaDual} are equal to $F(\stackrel{0\qquad }{\leftsink} \bullet \stackrel{\qquad n}{\rightsink})$.  So \eqref{TrKappa} and \eqref{TrKappaDual} must hold.  The snake relations are similarly checked by composing the corresponding morphisms in $\Trop$ and then applying $F$.  The conditions in \eqref{Tr-tau} are equivalent to the requirement that $F$ must intertwine the braidings of $\Trop$ and $\f{C}$.

It remains to show that such data suffices to completely determine a TrQFT.  For any tropical cobordism $\Gamma\in \Hom(\?{\Delta}_1,\?{\Delta}_2)$, we can always assume the following by inserting new edges $E$ of direction $0$: if an edge $E\in \Gamma^{[1]}$ has nonzero direction, and if $V\in \Gamma^{[0]}$ is a vertex of $E$, then $V$ is bivalent, and the other edge containing $V$ has direction $0$ (i.e., we insert a new compact direction-$0$ edge in the middle of every flag of $\Gamma$ for which the edge has nonzero direction).  Now pick an arbitrary tropical flow on $\Gamma$.  After possibly inserting additional direction-$0$ edges, we can assume that no two sources/sinks are contained in adjacent edges of $\Gamma$.  
Now if a source or sink is contained in a bivalent vertex $V\in \Gamma^{[0]}$, we modify the flow by moving this source or sink to either one of the adjacent vertices of $\Gamma'$ so it is no longer on a vertex of $\Gamma^{[0]}$.  The resulting flow has the property that all sources and sinks are in $(\Gamma')^{[0]}\setminus \Gamma^{[0]}$ or at vertices whose adjacent edges all have direction $0$.  This tropical flow now determines a decomposition of $\Gamma$ into morphisms which correspond under $F$ to the morphisms $\Tr_n$, $\Tr_n^{\vee}$, $\kappa_n$, $\kappa_n^{\vee}$, along with endomorphisms of $\s{C}_0$ which correspond to the Frobenius algebra operations.  Composing these yields the desired morphism $F(\Gamma):F(\?{\Delta}_1)\rar F(\?{\Delta}_2)$.

We next show that such data suffices to completely determine a TrQFT.  The idea is to show that for any tropical cobordism $\Gamma\in \Hom(\Delta_1,\Delta_2)$, a choice of tropical flow on $\Gamma$ induces a decomposition of $F(\Gamma)$ into compositions of tensors of the morphisms of the form $\id_n$, $\Tr_n$, $\Tr_n^{\vee}$, $\kappa_n$, and $\kappa_n^{\vee}$, plus the braiding morphisms $\tau_{\s{C}_n,\s{C}_{n'}}$ and morphisms corresponding to the other Frobenius operations on $\s{C}_0$.  This will indeed be possible after using the equivalence relation from Def. \ref{TropCobDfn} to insert some new edges of direction $0$.

More precisely, we insert direction-$0$ edges so that if an edge $E\in \Gamma^{[1]}$ has nonzero direction, and if $V\in \Gamma^{[0]}$ is a vertex of $E$, then $V$ is bivalent, and the other edge containing $V$ has direction $0$.  I.e., we insert a new compact direction-$0$ edge in the middle of every flag $(V,E)\in \Flags(\Gamma)$ for which $u(V,E)\neq 0$.  In pictures, $$
\bullet\stackrel{n}{\edge} \qquad \rightsquigarrow \qquad \bullet\stackrel{0}{\edge} \bullet \stackrel{n}{\edge}$$
where the element written above an edge indicates the flag's direction.

Now pick an arbitrary tropical flow on $\Gamma$.  After possibly inserting additional direction-$0$ edges, we can assume that no two sources/sinks are contained in adjacent edges of $\Gamma$.  Now if a source or sink is contained in a bivalent vertex $V\in \Gamma^{[0]}$, we modify the flow by moving this source or sink to either one of the adjacent vertices of $\Gamma'$ (for $\Gamma'$ as in Def. \eqref{flowdef}) so it is no longer on a vertex of $\Gamma^{[0]}$.  The resulting flow has the property that all sources and sinks are in $(\Gamma')^{[0]}\setminus \Gamma^{[0]}$ or at vertices whose adjacent edges all have direction $0$.  It is now straightforward to see that this tropical flow indeed determines (up to re-orderings/braidings) a decomposition of $\Gamma$ into tensors of morphisms of the form $\id_n$, $\Tr_n$, $\Tr_n^{\vee}$, $\kappa_n$, and $\kappa_n^{\vee}$, plus the braiding morphisms $\tau_{\s{C}_n,\s{C}_{n'}}$ and endomorphisms of $\s{C}_0$ which correspond to the Frobenius algebra operations (as in Example \ref{forgetEx}).  Composing these yields the desired morphism $F(\Gamma):F(\Delta_1)\rar F(\Delta_2)$.

We must now check that $F(\Gamma)$ does not depend on the choice of flow. The key observation is that one can interpret \eqref{TrKappa} as saying that we can move sinks past bivalent vertices, and one can interpret \eqref{TrKappaDual} saying that we can move sources past bivalent vertices.  If $\Gamma_0$ is a subgraph (with half-edges) of $\Gamma$ such that every edge has direction $0$, then the usual correspondence between commutative Frobenius objects and TQFT's ensures that sinks and sources can be freely moved around within $\Gamma_0$.  Furthermore, possibly after inserting more direction-$0$ edges into $\Gamma_0$, we can use the Frobenius relations to insert new sources and sinks in $\Gamma_0$, and then these can be moved to elsewhere in $\Gamma$ using \eqref{TrKappa} and \eqref{TrKappaDual}.  One sees that any two tropical flows on $\Gamma$ can be related by these operations.

To see that $F$ is compatible with compositions, observe as above that all compositions can be decomposed into those contained in subgraphs of weight $0$ edges, plus compositions with graphs of the form $\kappa_n$, $\kappa_n^{\vee}$, $\id_n$, $\Tr_n$, and $\Tr_n^{\vee}$, along with the braiding morphisms.  Since \eqref{TrKappa} and \eqref{TrKappaDual} let us move sources and sinks, we see that the only non-trivial relations which are not already generated by \eqref{TrKappa} and \eqref{TrKappaDual} are those generated by the snake relations.

Finally, compatibility with the equivalence relation of Definition \ref{TropCobDfn} is clear because contracting compact direction-$0$ edges just corresponds to contracting cylinders in the TQFT associated to $\s{C}_0$.  Thus, the data indeed determines a TrQFT.
\end{proof}

We note that when $\f{C}$ is a category of (super) $R$-modules for any commutative ring $R$, the snake relations imply that $\Tr_n$ induces an injection $\s{C}_n\hookrightarrow \s{C}_{-n}^{\vee}$ for each $n$.  In particular, $\f{C}_n$ and $\f{C}_{-n}$ must have the same rank, and if this rank is finite, the above injection must be an isomorphism, i.e., $\Tr_n$ is a perfect pairing.  In this case, dualizing $\Tr_{-n}$ and composing with these isomorphisms $\s{C}_{n}^{\vee}\risom  \s{C}_{-n}$ and $\s{C}_{-n}^{\vee} \risom \s{C}_n$ uniquely determines the ``co-trace'' $\Tr_n^{\vee}$.  Explicitly,  given $z\in R$, $\Tr_n^{\vee}(z)$ is the unique element $x\otimes y \in \s{C}_{n}\otimes\s{C}_{-n}$
such that
\begin{align}\label{CotraceDef}
    \Tr_{-n}(a\otimes x)\otimes y = a
\end{align}
for all $a\in \s{C}_{-n}$.  Given this duality between traces and co-traces, we can simplify our conditions somewhat: noting that condition \eqref{TrKappa} can be phrased as adjointness with respect to the traces, we see that Condition \eqref{TrKappaDual} is the dual statement (co-adjointness with respect to the co-traces) and thus follows from \eqref{TrKappa} automatically by dualizing.

For the TrQFT's which we shall consider, we will always have $F([n])=F([-n])$.  Sufficient data for defining some such TrQFT's can be given as follows: recall (cf. \cite[\S 3.6.8]{KockTQFT}) that a commutative Frobenius object includes the data of a product $\wedge$, a coproduct $\vee$, a unit $\eta$, and a counit $\epsilon$.  This data induces a trace $\Tr\coloneqq \epsilon \circ \wedge$ and a co-trace $\Tr^{\vee}\coloneqq \vee \circ \eta$.  Thus, one natural way to get the data of the maps $\Tr_n$ and $\Tr_n^{\vee}$ is by realizing them as the trace and co-trace of a commutative Frobenius object structure on $\s{C}_n=\s{C}_{-n}$. The trace and co-trace are then dual to each other as in the finite-rank super $R$-module cases above, and so we again have that Condition \eqref{TrKappa} automatically implies Condition \eqref{TrKappaDual} by dualizing.  Furthermore, the compatibility with the braiding as in \eqref{Tr-tau} is also automatic now from the commutativity of the Frobenius structure.  We thus see the following:

\begin{cor}\label{Sufficient}
Let $\?{N}=N/\{\pm \id\}$ be the set obtained by identifying $n$ with $-n$ for each $n$ in $N$.  The following data is sufficient to give a TrQFT $F:\Trop \rar \f{C}$:
\begin{itemize}
   \item For each $n\in \?{N}$, a commutative Frobenius object $\s{C}_n\in \f{C}$,
   \item A morphism $\kappa_n:\s{C}_n \rar \s{C}_0$ for each $n$ (the identity if $n=0$), and 
   \item A morphism $\kappa_n^{\vee}:\s{C}_0 \rar \s{C}_n$ which is adjoint to $\kappa_n$ with respect to the Frobenius traces (meaning it satisfies \eqref{TrKappa}).
\end{itemize}
\end{cor}

We will use Corollary \ref{Sufficient} to construct the TrQFT's of interest to us here.

\begin{rmk}\label{logTQFT}
As noted in Example \ref{forgetEx}, we have a forgetful morphism $\forget:\Trop\rar 2\Cob$.  On the other hand, $\Trop$ is roughly a version of $2\Cob$ in which the circles are colored by elements of $N$ (indeed, $\Trop$ is a colored PROP with $N$ as the set of colors), and in which some cylinders (corresponding to positive-weight edges) may act non-trivially.  To explain why these cylinders/edges should be allowed to act non-trivially, we suggest that the cylinder associated to a half-edge of nonzero weighted direction $n$ should be viewed as being semi-infinite, with boundary circle living on a toric divisor $D_n$ at infinity.  Indeed, this accurately describes the log curves whose tropicalizations are $\Gamma$.  The Frobenius algebras $\s{C}_n$ which we will define below can be viewed as (extensions of) the spaces of incidence conditions which one can impose on these punctures at infinity.
\end{rmk}

\subsection{Defining the Multiplicity TrQFT}\label{TrQFT_Const}

For each $n\in N$, let $M_n\coloneqq n^{\perp}\subset M$.  Note that $M_n$ is not affected by replacing $n$ with $-n$, so we can use the projection $\?{n}\in \?{N}$.  We take
\begin{align*}
    \s{C}_{\?{n}}\coloneqq \s{C}_{M_{n}},
\end{align*}
where $\s{C}_{M_{n}}$ is the graded-commutative Frobenius algebra $\Lambda^*(M_{n}\oplus M_{n})$ as in \eqref{CL}.  Recall that the trace is defined by projecting onto $\Lambda^{\topp}(M_n\otimes M_n)$ and identifying this with $\bb{Z}$ by taking $\Theta_{{n}}^{\Box}$ to $1$, where $\Theta_{{n}}$ is either choice of primitive element of $\Lambda^{\topp} M_{n}$.

The inclusion $M_{n}\hookrightarrow M$ induces an inclusion of graded-commutative algebras $i_{\?{n}}:\s{C}_{\?{n}}\hookrightarrow \s{C}_0$ (not respecting the traces).  We define $\kappa_0\coloneqq \id$, and for $\?{n}\neq 0$, we define $$\kappa_{\?{n}}=|{n}|^2i_{\?{n}}:\s{C}_{\?{n}}\hookrightarrow \s{C}_0,$$ where ${n}$ is either lift of $\?{n}$ (cf. Notation \ref{TropNotation}).  We will from now on identify $\s{C}_{\?{n}}$ with its image under $\kappa_{\?{n}}$, i.e., for nonzero $\?{n}$, \begin{equation}
\label{eq-include-n}
\s{C}_{\?{n}}=\iota_{(n,0)\wedge (0,n)}(\Lambda^{*} (M\oplus M)).
\end{equation} 

Finally, we define the adjoint maps: $\kappa^{\vee}_{0}\coloneqq \Id$, and for $\?{n}\neq 0$, 
\begin{align*}
    \kappa^{\vee}_{\?{n}}:\s{C}_0 &\rar \s{C}_{\?{n}} \\
    a &\mapsto \iota_{(n,0)\wedge (0,n)}(a).
\end{align*}
To check that these are indeed adjoints, first note that, with \eqref{eq-include-n} understood, $\Theta^{\Box}_{{n}} = \iota_{(n,0)\wedge (0,n)}(\Theta_0^{\Box})$.  Given $a\in \s{C}_{{0}}, b\in \s{C}_{\?{n}}$, set $k=\Tr_{0}(a,\kappa_{\?{n}}(b))$, i.e., $a\wedge \kappa_{\?{n}}(b) = k\Theta_0^{\Box}$.  Then we have
\begin{align*}
    \kappa_{\?{n}}^{\vee}(a)\wedge b &= \iota_{(n,0)\wedge (0,n)}(a)\wedge b \\
                                     &= \iota_{(n,0)\wedge (0,n)}(a\wedge \kappa_{\?{n}}(b)) \\
                                     &= \iota_{(n,0)\wedge (0,n)}(k\Theta_0^{\Box}) \\&= k\Theta_{{n}}^{\Box} 
\end{align*}
hence $\Tr_{\?{n}}(\kappa_{\?{n}}^{\vee}(a), b)= k$, as desired.  We thus obtain a TrQFT via Corollary~\ref{Sufficient}.  We denote the corresponding functor by 
\begin{align*}
F_{\Mult}:\Trop\rar s\bb{Z}\Mod.    
\end{align*}

\subsection{The Main Theorem}

Now let $\Gamma$ be a rigid tropical curve in $\f{T}_{g,\Delta}(\A,\Psi)$.  Its degree determines an object $(I,\?{\Delta})\in \Trop$, and $\Gamma$ (up to type and the negation-action) can be viewed as a morphism in $\Hom((I,\?{\Delta}),(\emptyset,~))$.  Applying a TrQFT $F$, we have
\begin{align*}
    F(\Gamma)\in \Hom\left(\bigotimes_{i\in I} \s{C}_{{\Delta(i)}}, \bb{Z}\right).
\end{align*}

For each $i\in I$, we have an affine incidence condition $A_i\subset N_{\bb{R}}$, say with weight $w_i$.  We take $\alpha_i$ to be an associated element of $\Lambda^* M$ as in Lemma \ref{Trop-Wedge}, that is,
\begin{align}\label{alphai}
    \alpha_i\in \Lambda^{\codim(A_i)} M \subset \Lambda^* M
\end{align} to be the unique-up-to-sign index $w_i$ element which restricts to $0$ on $\LL(A_i)$.  Note that since $\Delta(i)\in \LL(A_i)$, $\alpha_i$ is in fact contained in $\Lambda^* M_{\Delta(i)} \subset \Lambda^* M$, so $\alpha_i^{\Box}\in \s{C}_{{\Delta(i)}}$.
 Finally, define
\begin{align*}
    \gamma\coloneqq  \bigotimes_{i\in I} \alpha_i^{\Box} \in \bigotimes_{i\in I} \s{C}_{{\Delta(i)}}.
\end{align*}

\begin{thm}\label{MainThm}
For $\Gamma$ and $\gamma$ as above,
\begin{align}\label{MainThmEqn}
    (\Mult(\Gamma))^2 = (F_{\Mult}(\Gamma))(\gamma).
\end{align}
\end{thm}
\begin{proof}
Let $L$ be the lattice $\prod_{V\in \Gamma^{[0]}} N$.  Note that we can view $\Lambda^*(N\oplus N)$ as the cohomology ring of the torus $(N_{\bb{R}}\oplus N_{\bb{R}})/(N\oplus N)$, with Poincar\'e duality corresponding to the identification with the dual lattice induced by the Frobenius trace.  Furthermore, given a rational-slope subspace $A\subset N_{\bb{R}}$, the element $\alpha_A^{\Box}$ of Lemma \ref{Trop-Wedge} and \eqref{squaring} is the same as the element obtained by taking the Poincar\'e dual of the homology class of $\?{A\oplus A}\subset (N_{\bb{R}}\oplus N_{\bb{R}})/(N\oplus N)$.  The K\"unneth theorem gives us a graded isomorphism 
\begin{align}\label{KunnethL}
\Lambda^* (L\oplus L) \cong \bigotimes_{V\in \Gamma^{[0]}} \Lambda^*(N\oplus N).
\end{align}
Recall that the cup product on the cohomology of a space $X$ can be viewed as the map induced on cohomology by the composition of maps of cochain complexes
\begin{align*}
 C^*(X) \times   C^*(X) \rar C^*(X\times X) \stackrel{\Delta^*}{\rar} C^*(X)
\end{align*}
where the first map is the K\"unneth map, and the second map is the pullback by the diagonal map.  Our Frobenius algebra coproduct is obtained by applying Poincar\'e duality, applying the dual to the cup product, and then applying Poincar\'e duality again, so it follows that for each compact edge $E\in \Gamma$, the coproduct of $1\in \s{C}_{{u_E}}$ is equal to the Poincar\'e dual of the K\"unneth decomposion of the diagonal class in $H_*((N_{\bb{R}}/\bb{R}u_E\oplus N_{\bb{R}}/\bb{R}u_E)/(N/\bb{Z}u_E \oplus N/\bb{Z}u_E))$  Thus, $\vee(1)\in \s{C}_{u_E}\otimes \s{C}_{u_E}$ is the element corresponding to the tropical class $[\Delta_E]^{\Box}$.  Under the inclusion \eqref{eq-include-n}, $1\in \s{C}_{{u_E}}$ corresponds to $|u_E|^2$ in $\s{C}_0$, and so applying $\kappa_{u_E}\otimes \kappa_{u_E}$ followed by the isomorphism of \eqref{KunnethL}, we obtain the element associated to the tropical class $|u_E|^2([\Delta_E],0)\oplus (0,[\Delta_E])$ in $L_{\bb{R}}\oplus L_{\bb{R}}$.  Theorem \ref{MainThm} now follows using Lemma  \ref{Trop2Wedge} and Proposition \ref{MultTropIntersection} by treating the middle of every compact edge and the end of every non-compact edge (viewed as vertices of $\Gamma'$) as a source and treating every vertex of $\Gamma$ as a sink.
\end{proof}

\subsection{Explicit description of the coproduct}\label{coproduct}

Here we will clarify how a (graded)-commutative Frobenius $R$-algebra $A$ with product $\wedge$ and counit $\epsilon$ (so $\Tr=\epsilon\circ \wedge$) determines a coproduct $\vee$, and we express the coproduct of $1$ explicitly in our setup.\\[-2mm]

\noindent
\begin{minipage}[b]{0.6\linewidth}
\hspace{15pt}Given an element $z\in A$, the coproduct $\vee$ of $z$ is the unique element \begin{align*}
    \vee(z)=\sum_i x_i\otimes y_i \in A\otimes A
\end{align*} such that, for all $a\in A$,
\begin{align}\label{CoprodEqn}
    \sum_i \Tr(a\otimes x_i)\otimes y_i = a\wedge z
\end{align}
\end{minipage}
\hspace{0.5cm}
\begin{minipage}[b]{0.35\linewidth}
\centering
\includegraphics[width=\textwidth]{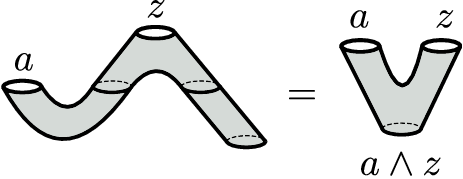}\\ \ \\
\end{minipage}\\
This is illustrated by the equivalence of the two cobordisms on the right.  We note that this agrees with \eqref{CotraceDef}, which was the case $z\in \eta(R)$.

Now let us specialize to our setup where $A=\s{C}_{\?{n}}$.  Let $e_1,\ldots,e_k$ be a basis for $\Lambda^*(M_n\oplus M_n)$ such that $e_1 \wedge \cdots \wedge e_k = \Theta_{{n}}^{\Box}$.  Given $I=\{i_1,\ldots,i_\ell\}\subset \{1,\ldots,k\}$ with $i_1<\ldots<i_{\ell}$, let $e_I\coloneqq e_{i_1}\wedge \cdots \wedge e_{i_{\ell}}$.  Then we claim that
\begin{align}\label{vee1}
    \vee(1) = \sum_{I_1\sqcup I_2 = \{1,\ldots k\}} (-1)^{\sign(I_2,I_1)} e_{I_1} \otimes e_{I_2},
\end{align}
where the sum is over all decompositions of $\{1,\ldots,k\}$ into disjoint subets $I_1$ and $I_2$, and $\sign(I_2,I_1)$ is the sign of the shuffle taking $(1,\ldots,k)$ to $(I_2,I_1)$. To check this, let $a=e_J$ for arbitrary $J\subset \{1,\ldots,k\}$, and consider
\begin{align*}
    \sum_{I_1\sqcup I_2=\{1,\ldots,k\}} (-1)^{\sign(I_2,I_1)} \Tr(e_J\otimes e_{I_1})  e_{I_2}.
\end{align*}
The factor $\Tr(e_J\otimes e_{I_1})$ is clearly nonzero if and only if $J=I_2$, and in this case we have $\Tr(e_J\wedge e_{I_1}) = (-1)^{\sign(I_2,I_1)}$, as desired.

Similarly, one sees directly from \eqref{CoprodEqn} that
\begin{align}\label{CoprodTheta}
    \vee(\Theta^{\Box}_n)=\Theta_n^{\Box}\otimes \Theta_n^{\Box}.
\end{align}

\subsection{A genus $1$ example}\label{gen1ex}

Consider a genus $1$ tropical curve in a plane as in Figure \ref{fig:g1-curve-split}.  Here, the three vertices $V_1,V_2,V_3$ are each $4$-valent, each contained in one contracted edge $E_i$, $i=1,2,3$ respectively (not pictured) satisfying a $\psi$-class condition and a line condition $A_i$ parallel to $v_i^{\perp}$ for $v_1=(b,-a)$, $v_2=(d,-c)$, and $v_3=(f,-e)$.  The conditions on the other non-compact edges are all taken to be trivial (i.e., corresponding to all of $N_{\bb{R}}$).  The edges all have weight $1$, and the directions of the edges are all determined by $u_{E_{12}}=(1,0)$, $u_{E_{13}}=(0,1)$, and $u_{E_{23}}=(1,-1)$.  We will illustrate our different methods for computing the multiplicity in this example.

\begin{figure}[htb]
    \centering
    \includegraphics[width=.96\textwidth]{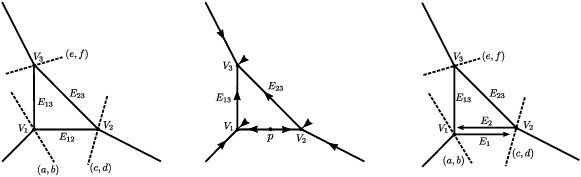}
    \caption{\textbf{Left}: A rigid genus one curve that has a $\psi$-class condition on each of three interior markings which also satisfy line-conditions (dashed). \textbf{Center}: Equipping the same curve with a flow (here the contracted edges are pictured as arrow-heads pointing to the vertices).  \textbf{Right}: Splitting $E_{12}$ at its midpoint $p$ and extending the resulting non-compact edges to infinity yields a non-rigid genus $0$ curve.
    \label{fig:g1-curve-split}}
\end{figure}

\subsubsection{Computing multiplicity using a determinant}

Using the standard basis $e_1,e_2$ for $N=\ZZ^2$, the map 
$$
\Phi:N_{V_1}\oplus N_{V_2}\oplus N_{V_3}
\ra 
\frac{N}{\ZZ u_{E_{12}}}\oplus \frac{N}{\ZZ u_{E_{13}}} \oplus \frac{N}{\ZZ u_{E_{23}}} \oplus \frac{N}{\ZZ (a,b)} \oplus \frac{N}{\ZZ (c,d)} \oplus \frac{N}{\ZZ (e,f)}
$$
as in \eqref{D} is given by the matrix that acts on row vectors
\begin{align*}
\Phi=
\left( \begin{matrix}
&1&&-b\\
1&&&a\\
&&1&&-d\\
-1&&1&&c\\
&-1&-1&&&-f\\
&&-1&&&e
\end{matrix}\right)
\end{align*}
and it has determinant $\det(\Phi)= ade+adf - bce-bde$.  Since each edge has weight $1$ and $\langle V_i\rangle=1$ for each $i=1,2,3$, this tells us that 
\begin{align}\label{MGamma}
    \Mult(\Gamma)&=|ad(e+f) - be(c+d)|.
\end{align}

\begin{rmk}
Since $ade+adf - bce-bde$ does not factor, there exists no formula for $\f{D}_{\Gamma}$ or  $\Mult(\Gamma)$ as a product of vertex multiplicities.  This is in contrast to planar tropical curves in the absence of $\psi$-classes \cite{Mi}.  Also, we will see in Corollary \ref{SplitCor} that for $\Gamma$ of genus zero, $\f{D}_{\Gamma}$ can always be expressed as a product of vertex multiplicities divided by a product of edge multiplicities, with the edge multiplicities always equaling $1$ in dimension $2$.  We note that this phenomenon of not factoring into a product of vertex multiplicities was also observed for refined elliptic tropical descendant invariants in \cite{SS}.
\end{rmk}

\subsubsection{Computing multiplicity using the TrQFT}\label{CompTrQFT}

We now demonstrate how this can be computed using the TrQFT approach of Theorem \ref{MainThm}.  We take $V_3$ to be a sink for our tropical flow, and for our sources we take the midpoint $p$ of $E_{12}$, along with all the non-compact edges.

At $p$, we have the associated Frobenius algebra $$\s{C}_{{e_1}}=\Lambda^*(M_{e_1}\oplus M_{e_1}),$$ where we recall that $M_{e_1}$ means $e_1^{\perp}=\bb{Z}\langle e_2^*\rangle \subset M$.  Consider the basis $f_1=(e_2^*,0)$, $f_2=(0,e_2^*)$ for $M_{e_1}$.  Then by \eqref{vee1}, we have $$\vee(1)=f_1\wedge f_2\otimes 1-f_1\otimes f_2 +f_2\otimes f_1 + 1\otimes f_1\wedge f_2$$ in $\s{C}_{{e}_1}$.  Let us denote the terms of this sum by $x_i\otimes y_i$, $i=1,2,3,4$, respectively.

Now, using the designated tropical flow and applying the operations from the construction of $F_{\Mult}$ to the incidence conditions and $\vee(1)$ as above,
we have that
\begin{align*}
    \Mult(\Gamma)^2 = \sum_{i=1}^4 \Tr\left[\iota_{(0,1)^{\Box}}(x_i \wedge v_1^{\Box}) \wedge v_3^{\Box} \wedge \iota_{(-1,1)^\Box}(y_i \wedge v_2^{\Box})  \right].
\end{align*}
One computes $\Theta_0^{\Box}=-e_1\wedge e_2\wedge e_3\wedge e_4$, and so the trace $\Tr$ is negative the determinant.  
One then computes the contributions from $i=1,2,3,4$ to be $b^2e^2(c+d)^2$, $-abde(c+d)(e+f)$, $-abde(c+d)(e+f)$ again, and $a^2d^2(e+f)^2$, respectively.  The resulting sum is indeed the square of the expression for $\Mult(\Gamma)$ from \eqref{MGamma}.

\subsubsection{Computing multiplicity using a splitting formula, up to signs}\label{SplitG1}

We note one more possible approach, employing a sort of splitting formula.  We use a flow as in the TrQFT approach \S \ref{CompTrQFT}, and we ``split'' the tropical curve at $p$, removing this midpoint and extending the two newly non-compact edges to infinity.  Label these $E_1$ and $E_2$ as in the right-most part of Figure \ref{fig:g1-curve-split}.

Now, let $A_1$ be an affine line passing near $p$ and parallel to $\bb{R}e_1$, and let $A_2=N_{\bb{R}}$.  If, in addition to the previously imposed conditions, we impose $A_1$ on $E_1$ and $A_2$ on $E_2$, the resulting rigid tropical curve has multiplicity $|be(c+d)|$.  If we instead impose $A_1$ on $E_2$ and $A_2$ on $E_1$, the resulting multiplicity is $|ad(e+f)|$.  These two can of course be combined, with some careful sign choices, to yield the multiplicity as given in \eqref{MGamma}.

Such a splitting is indeed always possible, and in the next section we will prove and apply this to genus $0$ cases.  The problem with higher-genus cases, as discussed further in Remark \ref{HigherGSplitting}, is that we do not have a nice general procedure for determining the correct signs when combining the multiplicities as above.  These sign issues are related to the necessity of the squaring trick employed in the construction of $F_{\Mult}$.

\subsection{A geometric interpretation of the squared lattices}\label{Speculation}

In our construction of $F_{\Mult}$, we replaced the lattices $M_n\coloneqq n^{\perp}\cap M$ with their squares $M_n\oplus M_n$, and we applied the squaring/diagonal operation $\Box$ defined in \eqref{squaring} to map simple elements of $\Lambda^* M_n$ to simple elements of $\Lambda^*(M_n\oplus M_n)$.  We offer here a geometric interpretation for this setup, along with some speculation on potential broader applications.

Let us take $\kk=\bb{C}$.  Recall that the tropical curve counts of \eqref{GWtrop} correspond to counts of log curves in a toric variety with cocharacter lattice $N$, hence dense torus orbit $N\otimes \bb{C}^*$.  We identify this with $T^* N_{\bb{R}}/N$, where the quotient is via the identification of $N$ with the lattice of integer cotangent vectors.  Alternatively, we may view this as $$(N_{\bb{R}}\oplus N_{\bb{R}})/(0,N).$$  We identify $M\oplus M$ with the dual of $N\oplus N$ in the natural way, so elements of $M\oplus M$ cut out linear subspaces of $N_{\bb{R}}\oplus N_{\bb{R}}$.  Simple elements of $\s{C}_0=\Lambda^*(M\oplus M)$ thus determine subspaces of $N\otimes \bb{C}^*$.  In particular, for $\alpha_A$ as in \eqref{alphaA}, the closure of the subspace cut out by $\alpha_A^{\Box}$ is indeed a representative of the algebraic cycle associated to the rational-slope affine-linear space $A\subset N_{\bb{R}}$.

More generally, let $D_n$ denote the dense torus orbit of the toric boundary stratum corresponding to a ray through $n$, or for $n=0$, let $D_0$ be the dense torus orbit $N\otimes \bb{C}^*$ considered above.  Then elements of $\Lambda^*(M_n\oplus M_n)$ as in \eqref{TrQFT_Const} cut out subspaces of $D_n$.  Again, $\alpha_{A_i}^{\Box}$ cuts out the algebraic subspace of $D_{\Delta(i)}$ associated to $A_i$.

One could imagine imposing conditions on, say, the norms or phases of marked points of log curves, and such conditions would correspond to elements of $\Lambda^* (M_n\oplus M_n)$ that are not of the form $\alpha^{\Box}$ for any $\alpha$.  It would be interesting to find tropical correspondence theorems allowing for such conditions and using our TrQFT to compute multiplicities.  Indeed, such conditions on phases appear in the work on log symplectic cohomology of Ganatra-Pomerleano \cite{GP,GP2}, which is still being further investigated by Gross-Pomerleano-Siebert \cite{GPomS}.

\section{A splitting formula for genus 0}\label{SplitSection}

\subsection{The Frobenius subalgebra of tropical classes}

For each $\?{n}\in N/\pm\id$, note that we have a Frobenius subalgebra\footnote{We say $A$ is a Frobenius subalgebra of $B$ if $A$ and $B$ are Frobenius algebras, $A$ is a subalgebra of $B$, and the Frobenius trace on $A$ is the restriction of the Frobenius trace on $B$.}
\begin{align*}
    \s{C}_{{n}}^{\Box}\subset \s{C}_{{n}}
\end{align*}
generated by elements of the form $\alpha^{\Box}$ for $\alpha$ a simple element of $\Lambda^* M_{n}$.  We refer these generators as the \textbf{tropical classes} since they are precisely the classes $\alpha_A^{\Box}$ for rational-slope affine linear subspaces $A\subset N_{\bb{R}}$.  Note that  $\s{C}_{{n}}^{\Box}$ is commutative, not just graded-commutative.  The maps $\kappa_{{n}}$ and $\kappa_{{n}}^{\vee}$ clearly restrict to well-defined adjoint maps between $\s{C}_{{n}}^{\Box}$ and $\s{C}_{0}^{\Box}$, and so by Proposition \ref{Sufficient}, we obtain a new TrQFT 
\begin{align*}
    F_{\Mult}^{\Box}:\Trop \rar \bb{Z}\Mod,
\end{align*}
with $\bb{Z}\Mod$ denoting the category of $\bb{Z}$-modules (not super $\bb{Z}$-modules).

Now suppose we have $\Gamma$ and $\gamma$ as in the setup of Theorem \ref{MainThm}.  The following says that the theorem still holds with $F_{\Mult}^{\Box}$ in place of $F_{\Mult}$ so long as $\Gamma$ is genus $0$.
\begin{thm}\label{FBoxThm}
   Suppose $\Gamma$ has genus $0$.  Then
   \begin{align*}
       (\Mult(\Gamma))^2 = (F_{\Mult}^{\Box}(\Gamma))(\gamma).
   \end{align*}
\end{thm}
\begin{proof}
Since $\Gamma$ has genus $0$, we can define a flow on $\Gamma$ with any choice of vertex as the unique sink.  For such a flow, every vertex other than the sink has a unique edge flowing out of it.  It follows that $F_{\Mult}(\Gamma)(\gamma)$ and $F_{\Mult}^{\Box}(\Gamma)(\gamma)$ can be computed using only the products, the maps $\kappa_{{n}}$ and $\kappa_{{n}}^{\vee}$, and a Frobenius trace at the sink (i.e., no coproducts are necessary).  Since $\s{C}_{{n}}^{\Box}$ is a Frobenius subalgebra of $\s{C}_{{n}}$ for each $\?{n}$, and since the maps $\kappa_{{n}}$ and $\kappa_{{n}}^{\vee}$ all respect the restrictions, it follows that $F_{\Mult}(\Gamma)(\gamma)=F_{\Mult}^{\Box}(\Gamma)(\gamma)$, as desired.
\end{proof}

Note that the coproducts on $\s{C}_{{n}}^{\Box}$ differ from the corresponding coproducts on $\s{C}_{{n}}$, and as a result, this argument fails in higher genus.  Indeed, we saw non-tropical classes with a non-trivial contribution to the multiplicities in the example of \S \ref{gen1ex}. There is, however, an exception for point conditions:
\begin{prop}\label{SplitPoint}
Let $\Gamma$ be a rigid tropical curve in $\f{T}_{g,\Delta}(\A,\Psi)$.  
Suppose for some $i\in I^{\circ}$, the incidence condition $A_i$ is just a point in $N_{\bb{R}}$.  
Let $V_i$ be the vertex in $E_i$, and let $I_i\coloneqq \{E\in \Gamma^{[1]}\setminus \{E_i\}:E\ni V_i\}$.  Let $\Gamma_i$ denote the tropical curve obtained from $\Gamma$ by forgetting the edge $E_i$ and vertex $V_i$, compactifying each $E\in I_i$ with a new vertex $V_E$, extending each $E\in I_i$ to infinity, and then attaching a new contracted edge $E^{\circ}$ to $V_E$ for each $E\in I_i$. We impose only the trivial condition $N_{\bb{R}}$ on the new unbounded edges $E$ for each $E\in I_i$, but on each of the new contracted edges $E^{\circ}$ we impose a point condition.  All other non-compact edges and vertices inherit conditions from the original $\A$ and $\Psi$ in the obvious way.  With these conditions on $\Gamma_i$, we have $\Mult(\Gamma)=\Mult(\Gamma_i)$.
\end{prop}
\begin{proof}
We use $F_{\Mult}$ and Theorem \ref{MainThm}.  Choose a tropical flow on $\Gamma_i$ for which $E_i$ flows into $V_i$, but all other edges of $I_i$ flow out of $V_i$.  Associated to the point condition on $E_i$ we have the element $\Theta_0^{\Box}$.  The flow through $V_i$ is then understood by repeatedly taking coproducts, and by \eqref{CoprodEqn}, this results in $\bigotimes_{E\in I_i} \Theta_0^{\Box}$.  These factors indeed correspond to imposing new point conditions as described in the statement of the proposition.  The result follows.
\end{proof}

\begin{rmk}\label{AlgKun}
Note that $\Lambda^*(N\oplus N)$ can be identified with the cohomology of the torus $T\coloneqq (N\otimes \bb{C})/(N\oplus iN)$ which is naturally a Frobenius algebra.  The tropical classes $\s{C}_{{n}}^{\Box}$ then correspond to algebraic classes in $H^*(T)$.  The fact that there is no splitting formula in general then corresponds to non-existence of an algebraic K\"unneth decomposition of diagonal classes for these Abelian varieties.  See also \cite[\S 4.3]{Rau} for further discussion on issues with splitting the diagonal class, along with another approach for circumventing this issue in genus $0$.
\end{rmk}

\subsection{Splitting formula for genus $0$}

We next give an explicit description for the coproduct $\vee$ on $\s{C}_{{n}}^{\Box}$.  Let $\{e_j\}_{j\in J}$ be a basis for $M_{n}$, indexed by a set $J$.  For $I=\{j_1,\ldots,j_k\}\subset J$, let $e_I\coloneqq e_{j_1}\wedge \cdots \wedge e_{j_k} \in \Lambda^* M_n$ (the sign will not matter).  In particular, $e_{\emptyset}\coloneqq 1$.  Then for $e_I^{\Box} \in \s{C}_{{n}}$, the reasoning used to compute \eqref{vee1} yields
\begin{align}\label{sqcop}
\vee(e_I^{\Box}) = \sum_{I_1\sqcup I_2=J\setminus I} e^{\Box}_{I_1} \otimes e^{\Box}_{I_2},
\end{align}
where the sum is over all decompositions of $J\setminus I$ into a disjoint pair of subsets $I_1$ and $I_2$.

Now, for an edge $E\in \Gamma^{[1]}$ with weighted direction $u_E$, let $u_E'$ denote the primitive vector with direction $u_E$.  Applying \eqref{sqcop} to $1\in \s{C}^{\Box}_{{u_E}}$ yields the following splitting formula:
\begin{thm}[Genus $0$ tropical splitting formula]\label{SplitThm}
   Let $\Gamma$ be a genus $0$ tropical curve satisfying a rigid collection of conditions $\A$, $\Psi$, and let $E$ be a compact edge of $\Gamma$.  Let $\Gamma_1$, $\Gamma_2$ be the two genus $0$ tropical curves obtained by splitting $\Gamma$ at $E$ and then extending the resulting half-edges to infinity.  Let $\Psi_i$ denote the $\psi$-class conditions induced on $\Gamma_i$ by $\Psi$ for $i=1,2$, respectively.
   
   Let $\{e_1,\ldots e_{r-1},u'_E\}$ be a basis for $N$.  Given $I\subset \{1,\ldots,r-1\}$, let $A_I\subset N_{\bb{R}}$ denote the affine space containing $E$ and spanned by $\{e_i\}_{i \in I} \cup \{u'_E\}$.  For $i=1,2$ let $\A_{i,I}$ denote the incidence conditions induced on $\Gamma_i$ by $\A$, with $A_I$ being the condition on the new unbounded edge extending $E$.  Then
   \begin{align}\label{SplitEqn}
       \Mult_{\A}(\Gamma)=w(E)\sum_{I_1\sqcup I_2=\{1,\ldots,r-1\}} \Mult_{\A_{1,I_1}}(\Gamma_1) \Mult_{\A_{2,I_2}}(\Gamma_2),
   \end{align}
   where the sum is over all decompositions of $\{1,\ldots r-1\}$ into a disjoint pair of subsets $I_1$ and $I_2$.
\end{thm}

\subsection{Vertex and edge multiplicities}

Now consider one of the tropical curves $\Gamma_i$ as in the above theorem, together with the conditions $\A_i,\Psi_i$ induced by $\A$ and $\Psi$, but with the condition on the new unbounded edge $E_i$ being trivial (so $A_{i,I}$ for $I$ the full set $\{1,\ldots,r-1\}$).  Then $\Gamma_i$ is not necessarily rigid.  In particular, there may be small deformations of $\Gamma_i$ which still satisfy $\A_i$ and $\Psi_i$ but have $E_i$ being translated from its original location.  These translations of $E_i$ sweep out a patch of an affine linear space, and we denote the corresponding linear space, intersected with $N$, by $W_{E_i}$. The main feature of these is that, by rigidity,
$$ W_{E_1}\cap W_{E_2}\subseteq \ZZ u_E.$$
Equivalently, if we treat $E_i$ as the lone outgoing edge of $\Gamma_i$ (with no other sinks), then for  $\gamma_i$ defined as in Theorem \ref{MainThm} for the conditions $\A_i$ on $\Gamma_i$, we can consider $\ker(F_{\Mult}^{\Box})\subset N\oplus N$.  Then for $p:N\oplus N\ra N$ the projection onto either factor, we have 
\begin{align*}
    W_{E_i}=p(\ker(F_{\Mult}^{\Box}(\gamma_i)))\subset N.
\end{align*}

Now when choosing the basis $e_1,\ldots,e_{r-1},u_E$ in Theorem \ref{SplitThm}, after taking a finite-index refinement $N_E$ of the lattice $N$, this basis can be chosen so that \begin{align*}
    A_{1,I_1} = W_{E_2} \mbox{\hspace{.4 in} and \hspace{.4 in} } A_{2,I_2} = W_{E_1}
\end{align*} for some choice of $I_1\sqcup I_2=\{1,\ldots,r-1\}$.  In this case, this will be the only choice of $I_1,I_2$ with a nonzero contribution to \eqref{SplitEqn}.  For the refinement $N_E$ here taken to be as small as possible, the index of $N$ in $N_E$ is
\begin{align*}
\Mult(E)\coloneqq \inde\Big(N/\bb{Z}u_E \rar (N/W_{E_1})\oplus (N/W_{E_2})\Big),
\end{align*}
called the \textbf{edge-multiplicity} of $E$.  Similarly, for each vertex, we define a \textbf{vertex-multiplicity}
\begin{align*}
   \Mult(V)\coloneqq  \inde\left(N \rar \prod_{E\ni V} N/W_{\partial_V E,E} \right)
\end{align*}
where $W_{\partial_V E,E}$ means $W_{E_i}$ associated to the component of $\Gamma\setminus \{E\}$ which does not contain $V$, and for non-compact edges $E_i$,  $W_{\partial_V E_i,E_i}\coloneqq A_i\cap N$.

Now, inductively applying Theorem \ref{SplitThm} to every compact edge and choosing our bases $\{e_i\}$ as above for each compact edge, we obtain the following:
\begin{cor}\label{SplitCor}
\begin{align*}
    \f{D}_{\Gamma} = \frac{\prod_{V\in \Gamma^{[0]}} \Mult(V)}{\prod_{E\in \Gamma^{[1]}_c} \Mult(E)}.
\end{align*}
\end{cor}

\begin{rmk}\label{HigherGSplitting}
As mentioned in \S \ref{SplitG1}, one could state a modified version of the splitting formula \eqref{SplitEqn} in higher-genus by equipping the affine subspaces $A_i$ with orientations and then using signed intersections of the tropical classes.  This version of the multiplicity calculation actually follows directly from the definition of $\f{D}_{\Gamma}$ as the absolute value of determinant of a matrix as in \eqref{DGamma}.  However, re-ordering these oriented versions of tropical cycles (corresponding to reordering columns of the matrix) -- e.g., when trying to group together conditions associated to the same vertex -- results in numerous sign changes, and this prevents one from writing a nice analog of Corollary \ref{SplitCor} using this approach.
\end{rmk}

\section{Multiplicities from brackets of polyvector fields}\label{LInfinitySection}

\subsection{Flows with a single sink}

In the proof of Theorem \ref{MainThm}, we used a tropical flow in which the midpoint of each compact edge of $\Gamma$ was a source, and each vertex was a sink.  However, the point of introducing the TrQFT formalism is that any other choice of tropical flow will produce a different method of computing the multiplicities.

In this section we consider the case of a rigid genus $0$ tropical curve $\Gamma$ equipped with a flow consisting of a single sink at a vertex $V_{\infty}$.  In this setting, we recursively associate elements $\alpha_E \in \Lambda^* M$ (determined up to sign) to each edge $E\in \Gamma^{[1]}$ as follows:
\begin{itemize}
    \item For each non-compact edge $E_i$, $i\in I$, we take $\alpha_{E_i}\coloneqq \alpha_{A_i}$ as defined in \eqref{alphaA}.
    \item Suppose $\{E_j\}_j$ are the edges flowing into a vertex $V\neq V_{\infty}$ and $E_{\out}$ is the edge flowing out of $V$.  Assume by induction that each $E_j$ has already been assigned some $\alpha_{E_j}\in \Lambda^* M$.  Let $n_{\out}$ be the weighted direction of $E_{\out}$ (pointing opposite the flow).  Then 
    \begin{align}
        \alpha_{E_{\out}}\coloneqq \iota_{n_{\out}}\left(\bigwedge_j \alpha_{E_j}\right).
    \end{align}
\end{itemize}
For $\alpha \in \Lambda^*M$, $\alpha^{\Box}$ as in \eqref{squaring}, $n\in N$, and $\kappa_{{n}}$ and $\kappa_{{n}}^{\vee}$ as in \S \ref{TrQFT_Const}, one checks that
\begin{align*}
    \kappa_{{n}}\circ \kappa_{{n}}^{\vee}(\alpha^{\Box}) = (-1)^{\deg(\alpha)-1}(\iota_n(\alpha))^{\Box}.
\end{align*}
Thus, for $\Gamma_{V}$ consisting of $V$, the half-edges $E_j$ flowing into $V$, and the half-edge $E_{\out}$ flowing out of $V$ (glued appropriately), we have
\begin{align*}
    F_{\Mult}(\Gamma_V)(\bigotimes_j \alpha_{E_j}^{\Box}) = \pm\alpha_{E_{\out}}^{\Box}.
\end{align*}
It now follows by induction and Theorem \ref{MainThm} that the multiplicity of $\Gamma$ is given by
\begin{align*}
    \Mult(\Gamma)=\left|\left\langle \Omega, \bigwedge_{E\ni V_{\infty}} \alpha_E \right\rangle\right|,
\end{align*}
where $\Omega$ is a primitive top-degree form in $\Lambda^* N$ and $\langle \cdot,\cdot \rangle$ denotes the dual pairing.

In the next subsection, we re-frame this construction in terms of mirror polyvector fields before stating this multiplicity formula as a theorem.

\subsection{Mirror polyvector fields and multiplicities}\label{MirrorPolySection}

Consider the algebra
\begin{align}\label{A}
    A\coloneqq \bb{Z}[N]\otimes \Lambda^* M.
\end{align}
This can be viewed as the algebra of integral \textbf{polyvector fields} on the algebraic torus $\bb{G}_m(M)$ dual/mirror to $\bb{G}_m(N)$ (significance to mirror symmetry will be discussed in \S \ref{MSapps}).  An element $z^n\otimes m\in \bb{Z}[N]\otimes M$ corresponds to the derivation 
\begin{align}\label{derivation}
    z^a \mapsto \langle a,m\rangle z^{a+n}
\end{align} of $\bb{Z}[N]=\Gamma(\bb{G}_m(M),\s{O}_{\bb{G}_m(M)})$.  We will often abbreviate the notation $z^n\otimes \alpha$ as simply $z^n \alpha$, and similarly, we will often write wedge-products $\alpha\wedge \beta$ as simply $\alpha\beta$.

We define a linear form $\ell_1:A\rar A$ by
\begin{align*}
\ell_1(z^n\alpha) \coloneqq   z^n \iota_n(\alpha),
\end{align*}
and furthermore, we define multilinear functions $\ell_k:A^{\otimes k}\rar A$ by
\begin{align}\label{lk}
    \ell_k(z^{n_1}\alpha_1,\ldots,z^{n_k}\alpha_k)\coloneqq \ell_1\left(\prod_{j=1}^k z^{n_j}\alpha_j \right)= z^{n_1+\ldots+n_k} \iota_{n_1+\ldots+n_k}(\alpha_1\wedge \cdots \wedge \alpha_k).
\end{align}
We will study the structure of these brackets $\ell_k$ in \S \ref{Linf}.  First, we restate the multiplicity computation from above in terms of these brackets:

\begin{thm}\label{BracketMult}
Given a rigid genus $0$ tropical curve $\Gamma\in\f{T}_{g,\Delta}(\A,\Psi)$ with a flow towards a specified sink $V_{\infty}$, we inductively associate an element, well-defined up to sign, 
\begin{align*}
\zeta_E\coloneqq z^{n_E} \otimes \alpha_E \in \bb{Z}[N]\otimes \Lambda^* M
\end{align*}
as follows:
\begin{itemize}
    \item For each $i\in I$, 
     take $\zeta_{E_i}\coloneqq z^{\Delta(i)}\otimes \alpha_{A_i}$ for $\alpha_{A_i}$ as defined in \eqref{alphaA}.
    \item Let $E_1,\ldots,E_k$ be the edges flowing into a vertex $V\neq V_{\infty}$, and let $E_{\out}$ be the edge flowing out of $V$.  We take
    \begin{align}\label{sbracket}
        \zeta_{E_{\out}}\coloneqq \ell_k(\zeta_{E_1},\ldots,\zeta_{E_s}).
    \end{align}
\end{itemize}
Let $\Omega$ be a primitive element of $\Lambda^r N$.  Then $\Mult(\Gamma)$ equals the absolute value of the dual pairing:
\begin{align}\label{ProdZetaE}
    \Mult(\Gamma)=\left|\left\langle \Omega,\prod_{E\ni V_{\infty}} \zeta_E \right\rangle \right|.
\end{align}
 \end{thm}
In the construction above, it follows from induction and the balancing condition that $n_E$ for each edge $E$ is the weighted direction of $E$ in the direction opposite that of the flow towards $V_{\infty}$.  Thus, $\sum_{E\ni V} n_E =0$ by the balancing condition, and then rigidity implies that  $\prod_{E\ni V_{\infty}} \zeta_E$ is in $\Lambda^r M$.  We note that \eqref{ProdZetaE} can alternatively be computed as the index of $\prod_{E\ni V_{\infty}} \zeta_E$ in $\Lambda^{\topp} M$. We also note that this index is the same as the absolute value of the integral from \cite[\S 4]{BK}.

\subsection{L-infinity, Gerstenhaber, and BV-structures}\label{Linf}

We next explore the structure of the algebra $A=\bb{Z}[N]\otimes \Lambda^* M$ from \eqref{A} and the $k$-brackets $\ell_k$ of \eqref{lk} (with a sign-modification), as well as some consequences of this structure.  We denote
\begin{align*}
    A_0\coloneqq \ker(\ell_1)\subset A,
\end{align*}
i.e., $A_0$ is the submodule generated over $\bb{Z}$ by elements of the form $z^n \alpha$ with $\iota_n(\alpha)=0$.  Note that $A_0$ is closed under the brackets $\ell_k$ for each $k$.  This subspace $A_0$ is especially important because it contains the elements of $A$ which can actually show up as some $\zeta_E$ in the multiplicity computations of Theorem \ref{BracketMult}.

\subsubsection{Grading}

Consider the grading $\deg$ on $A$ given by $\deg(z^n\otimes \alpha)\coloneqq d$ when $\alpha \in \Lambda^d M$.  This makes $A$ into a graded commutative algebra under the product $(z^{n_1}\alpha_1)\cdot (z^{n_2}\alpha_2) = z^{n_1+n_2}\alpha_1\wedge \alpha_2$.  That is,
\begin{align}\label{GradeComm}
    (z^{n_1}\otimes \alpha_1)\cdot (z^{n_2}\otimes \alpha_2) = (-1)^{\deg(\alpha_1)\deg(\alpha_2)} (z^{n_2}\otimes \alpha_2)\cdot (z^{n_1}\otimes \alpha_1).
\end{align}
We let $|\cdot|$ denote the grading associated to $A[-1]$, i.e., 
\begin{align*}
    |z^n\otimes \alpha|\coloneqq \deg(z^n\otimes \alpha)-1.
\end{align*}
Given homogeneous elements $\zeta_1,\ldots,\zeta_k\in A$, we denote
\begin{align}\label{epsilon}
    \epsilon(\zeta_1,\ldots,\zeta_k)\coloneqq (-1)^{\sum_{i=1}^k (k-i)|\zeta_i|}.
\end{align}

\subsubsection{The sign-modified bracket $l_k$}\label{sec:lk}

Recall that for the sake of computing multiplicities, the $k$-brackets $\ell_k$ of \eqref{lk} only matter up to sign.  Thus, Theorem \ref{BracketMult} remains unchanged if we replace the brackets $\ell_k$ with the modified brackets $l_k:A^{\otimes k} \rar A$ defined on homogeneous elements by
\begin{align*}
    l_k(\zeta_1,\ldots,\zeta_k)\coloneqq  \epsilon(\zeta_1,\ldots,\zeta_k) \ell_k(\zeta_1,\ldots,\zeta_k).
\end{align*}
Note that $A_0=\ker(l_1)$, and that $A_0$ is closed under $l_k$ for each $k$.

One easily sees that $l_k$ has degree $-1$ under $\deg$ and degree $k-2$ under $|\cdot|$, i.e.,
\begin{align}\label{degree k-2}
    \deg[l_k(\zeta_1,\ldots,\zeta_k)]&=\left(\sum_{i=1}^k \deg(\zeta_k)\right)-1, \nonumber\\
    \left|l_k(\zeta_1,\ldots,\zeta_k) \right|&=\left(\sum_{i=1}^k |\zeta_k|\right)+(k-2).
\end{align}

\subsubsection{Graded skew symmetry of $l_k$}\label{GradedSymSection}

Now let $\sigma\in S_k$ be a permutation of homogeneous elements $(\zeta_1,\ldots,\zeta_k)$, and let $\chi(\sigma,\zeta_1,\ldots,\zeta_k)\in \{\pm 1\}$ denote the \textbf{graded signature}, meaning the product of the ordinary signature of $\sigma$ with a factor of $(-1)^{|\zeta_i||\zeta_j|}$ for each transposition of adjacent entries $\zeta_i,\zeta_j$ in a decomposition of the permutation as a product of such transpositions.  Equivalently, since $|\zeta_i||\zeta_j|=\deg(\zeta_i)\deg(\zeta_j)-\deg(\zeta_i)-\deg(\zeta_j)+1$, we see using \eqref{GradeComm} that $\chi(\sigma,\zeta_1,\ldots,\zeta_k)$ is determined by
\begin{align*}
    \epsilon(\zeta_{\sigma(1)},\ldots,\zeta_{\sigma(k)})\zeta_{\sigma(1)}\cdots \zeta_{\sigma(k)} = \chi(\sigma,\zeta_1,\ldots,\zeta_k)\epsilon(\zeta_1,\ldots,\zeta_k)\zeta_1 \cdots \zeta_k,
\end{align*}
or equivalently, writing $\zeta_i=z^{n_i}\alpha_i$ for each $i$,
\begin{align}\label{chiomega}
    \epsilon(\zeta_{\sigma(1)},\ldots,\zeta_{\sigma(k)})\alpha_{\sigma(1)}\wedge\cdots\wedge \alpha_{\sigma(k)} = \chi(\sigma,\zeta_1,\ldots,\zeta_k)\epsilon(\zeta_1,\ldots,\zeta_k)\alpha_1 \wedge\cdots\wedge \alpha_k.
\end{align}
Hence,
\begin{align}\label{GradedSkew}
   l_k(\zeta_{\sigma(1)},\ldots,\zeta_{\sigma(k)}) = \chi(\sigma,\zeta_1,\ldots,\zeta_k)l_k(\zeta_1,\ldots,\zeta_k).
\end{align}

\subsubsection{The bracket $l_2$ as the Schouten-Nijenhuis bracket}\label{l2SN}

 We recall the standard Schouten-Nijenhuis bracket $[\cdot,\cdot]$ on $A$, i.e., the unique  extension of the Lie bracket/Lie derivative to a graded bracket making $A$ into a Gerstenhaber algebra.  It can be defined as follows.  For $a_0,\ldots,a_k,b_0,\ldots,b_{\ell}\in \bb{Z}[N]\otimes M$, one defines
 \begin{align}\label{SN}
     [a_0\cdots a_k,b_0\cdots b_{\ell}] = \sum_{i,j}(-1)^{i+j}[a_i,b_j] a_0\cdots \wh{a_i} \cdots a_k b_0 \cdots \wh{b_j} \cdots b_{\ell},
 \end{align}
 where the hat indicates omission of the element, and where $[a_i,b_j]$ is the usual Lie bracket of the corresponding vector fields.  E.g., for $n_1,n_2\in N$ and $m_1,m_2\in M$, 
\begin{equation} \label{eq-usual-Lie}
 [z^{n_1}m_1,z^{n_2}m_2]\coloneqq z^{n_1+n_2}(\iota_{n_2}(m_1) m_2 - \iota_{n_1}(m_2) m_1).
\end{equation} 
 This is extended to include degree $0$ elements by defining $[z^{n_1},z^{n_2}]=0$, and for $\alpha \in \Lambda^* M$, $$[z^{n_1}\alpha,z^{n_2}]\coloneqq z^{n_1+n_2} \iota_{n_2}(\alpha).$$
For $\alpha,\beta\in \Lambda^* M$, $n_1,n_2\in N$ one has by \eqref{SN},
$$[z^{n_1}\alpha,z^{n_2}\beta]=-(-1)^{|\alpha||\beta|}[z^{n_2}\beta,z^{n_1}\alpha].$$
For $n\in N$, $\alpha,\beta\in  \Lambda^* M$, we find the special cases
\begin{equation} \label{eq-basic-Lie-props}
    [\alpha,\beta]=0,
    \qquad [z^n\alpha,\beta]=-z^n\alpha\iota_n(\beta),
    \qquad [\alpha,z^n\beta]=(-1)^{|\alpha|}z^n\iota_n(\alpha)\beta.
\end{equation}

Now, given $n_1,n_2\in N$ and $\alpha_0,\ldots,\alpha_k,\beta_0,\ldots,\beta_{\ell}\in M$, we apply \eqref{SN} to the case where $a_0=z^{n_1}\alpha_0$, $b_0=z^{n_2}\beta_0$, $a_i=\alpha_i$ for $i=1,\ldots,k$, and $b_j=\beta_j$ for $j=1,\ldots,\ell$.   Denoting $\alpha = \alpha_0\cdots \alpha_k$ and $\beta=\beta_0\cdots \beta_{\ell}$, we obtain using \eqref{SN} and \eqref{eq-usual-Lie}:
 \begin{align}
     [z^{n_1}\alpha,z^{n_2}\beta] =& [(z^{n_1}\alpha_0)\alpha_1\cdots\alpha_k,(z^{n_2}\beta_0)\beta_1\cdots\beta_{\ell}] \nonumber \\ =& \nonumber z^{n_1+n_2} \left[\left(\sum_{i=0}^k (-1)^{i}\big( \iota_{n_2}(\alpha_i)\beta_0\big) \alpha_0\cdots \wh{\alpha_i}\cdots \alpha_k\beta_1\cdots \beta_{\ell} \right) \right.\\
     &\left. - \left(\sum_{j=0}^{\ell} (-1)^{j} \big(\iota_{n_1}(\beta_j) \alpha_0\big)\alpha_1 \cdots \alpha_k \beta_0\cdots \wh{\beta_j}\cdots \beta_{\ell} \right)\right] \nonumber \\
     =& (-1)^k z^{n_1+n_2}\left(\iota_{n_2}(\alpha)\beta + (-1)^{k+1} \alpha\iota_{n_1}(\beta)\right). \label{schouten0}
 \end{align}
Here, we implicitly assumed that $\deg(\alpha)> 0$ and $\deg(\beta)> 0$, but one easily checks that \eqref{schouten0} extends to the $\deg=0$ cases as well.

On the other hand, after computing  $\epsilon(\alpha,\beta)=(-1)^k$, we see that 
 \begin{align}\label{l2}
     l_2((z^{n_1}\alpha_0)\alpha_1\cdots\alpha_k,(z^{n_2}\beta_0)\beta_1\cdots\beta_{\ell})&=(-1)^{k}z^{n_1+n_2}\iota_{n_1+n_2}(\alpha\beta) \nonumber \\
     &= (-1)^{k}z^{n_1+n_2}\left(\iota_{n_1+n_2}(\alpha)\beta +(-1)^{k+1}\alpha\iota_{n_1+n_2}(\beta)\right) 
 \end{align}
Note that \eqref{schouten0} and \eqref{l2} agree when $z^{n_1}\alpha$ and $z^{n_2}\beta$ are both contained in $A_0$.  
We have thus proven the following:

\begin{prop}\label{Schouten}
Then bracket $l_2$ agrees with the Schouten-Nijenhuis bracket $[\cdot,\cdot]$ on $A_0$.
\end{prop}

\subsubsection{$l_1$ as a BV-operator}\label{l1BV}
Recall that a \textbf{BV-algebra} (Batalin-Vilkovisky algebra) is the data of an associative graded commutative algebra $\s{A}$ together with a degree $(-1)$ unary linear operator $\delta:\s{A}\rar \s{A}$ such that $\delta\circ \delta=0$ and such that, for all homogeneous $a,b,c\in A$, one has
\begin{align}\label{7term}
    \delta(abc)=&\delta(ab)c+(-1)^{\deg(a)} a\delta(bc)+(-1)^{\deg(b)(\deg(a)+1)}b\delta(ac) \\&- \delta(a)bc-(-1)^{\deg(a)}a\delta(b)c - (-1)^{\deg(a)+\deg(b)}ab\delta(c).\nonumber
\end{align}
Consider our algebra $A\coloneqq \bb{Z}[N]\otimes \Lambda^* M$.  This is of course an associative graded commutative algebra with $\deg$ as the grading.  Furthermore, the operator $l_1$ is easily seen to have degree $(-1)$ and satisfy $l_1\circ l_1=0$, and a straightforward calculation reveals that it also satisfies \eqref{7term}.  Thus, $A$ together with the operator $l_1$ is a BV-algebra.

A standard property of BV-algebras is that they canonically admit a bracket $[\cdot,\cdot]$ making them into Gerstenhaber algebras.  This bracket is defined as the failure of $\delta$ to be a derivation, i.e.,
\begin{align}\label{DerivationFail}
[a,b]=(-1)^{\deg(a)}\delta(ab) - (-1)^{\deg(a)} \delta(a)b - a \delta(b).
\end{align}
Furthermore, it follows that $\delta$ gives a derivation for the bracket, i.e.,
\begin{align*}
    \delta([a,b])=[\delta(a),b]+(-1)^{\deg(a)-1}[a,\delta(b)].
\end{align*}

We now check that the bracket determined by \eqref{DerivationFail} for $A$ and $\delta$ is negative the Schouten-Nijenhuis bracket.  Let $a=z^{n_1}\alpha$ and $b=z^{n_2}\beta$ with $\deg(\alpha)=k+1$, $\deg(\beta)=\ell+1$.  Using \eqref{schouten0}, we compute
\begin{align*}
l_1(ab)=l_1((z^{n_1}\alpha)(z^{n_2}\beta)) &= z^{n_1+n_2}\left(\iota_{n_1+n_2}(\alpha)\beta + (-1)^{k+1} \alpha \iota_{n_1+n_2}\beta \right) \\
&= (-1)^k[a,b] + l_1(a)b+ (-1)^{k+1}al_1(b),
\end{align*}
where the bracket $[\cdot,\cdot]$ is the Schouten-Nijenhuis bracket.  The claim \eqref{DerivationFail} now follows for negative this bracket by rearranging the terms.  We have thus proven:

\begin{prop}\label{BVprop}
$A$ is a BV-algebra with $l_1$ as the BV-operator, and the associated bracket is negative the Schouten-Nijenhuis bracket. 
\end{prop}

\subsubsection{$l_1$ as the pullback of the differential}\label{l1Delta}

We next offer another interpretation of $l_1$, relating it to the operator $\Delta$ of \cite[\S 2.1]{BK} (in which \eqref{DerivationFail} is interpreted as the Tian-Todorov lemma).  Choose a primitive element $\Omega$ of $\Lambda^n N$.  This gives an isomorphism
\begin{align*}
    \bb{Z}[N]\otimes \Lambda^* M \risom \bb{Z}[N]\otimes \Lambda^* N, \hspace{.25 in}
    \omega \mapsto \iota_{\omega} \Omega.
\end{align*}
Let $d$ denote the exterior differential on $\bb{Z}[N]\otimes \Lambda^* N = \Omega^*(\bb{G}_m(M))$, the space of differential forms on $\bb{G}_m(M)$.  In other words, 
\begin{align}\label{dFormula1}
    d(z^n\otimes \xi)=z^n\otimes (n\wedge \xi).
\end{align} 

One defines $\Delta$ by the formula
\begin{align}\label{DeltaDef1}
    \iota_{\Delta(z^n \omega)} \Omega = d(\iota_{z^n\omega} \Omega).
\end{align}

It is claimed in \cite[\S 2.1]{BK} that $\Delta$ is a BV-operator and satisfies \eqref{DerivationFail} as $\delta$ when $[,]$ is taken to be the Schouten-Nijenhuis bracket, however the signs are off by Proposition \ref{BVprop} combined with the following result.

\begin{prop}\label{BVProp} For any homogeneous $\zeta\in A$, $\Delta(\zeta) = (-1)^{\deg(\zeta)+1}l_1(\zeta)$.
\end{prop}
We note that this sign is the difference between using left-contraction and right-contraction when defining $l_1$.
\begin{proof}
We can assume $\zeta$ has the form $z^n \omega$ for $\omega \in \Lambda^* M$ homogeneous.  From the definition of $l_1$, the claim is that
\begin{align}\label{DeltaFormula1}
    \Delta (z^n\omega) = (-1)^{\deg(\omega)+1}z^n \iota_n(\omega).
\end{align}
Substituting this into the left-hand side of \eqref{DeltaDef1} yields $(-1)^{\deg(\omega)+1} z^n \iota_{\iota_n \omega}(\Omega)$, while \eqref{dFormula1} makes the right-hand side of \eqref{DeltaDef1} into $z^n(n\wedge \iota_{\omega} \Omega)$.  Dividing both sides by $z^n$, the claim reduces to showing that 
\begin{align}\label{DeltaFormulaReduced1}
    (-1)^{\deg(\omega)+1} \iota_{\iota_n \omega}(\Omega) = n\wedge \iota_{\omega}(\Omega).
\end{align}
Since each side is contained in $\Lambda^{r-\deg(\omega)+1}N$, it suffices to check that they both give the same function on the dual space $\Lambda^{r-\deg(\omega)+1} M$.  Let
$$\langle \cdot,\cdot\rangle:(\Lambda^{r-\deg(\omega)+1}M)\otimes (\Lambda^{r-\deg(\omega)+1} N)\rar \bb{Z}$$
denote the dual pairing.  Let $\alpha\in \Lambda^{r-\deg(\omega)+1} M$.  For any $\beta\in\Lambda^{r-\deg(\omega)+1}N$, $\gamma\in \Lambda^* N$, and $\delta\in \Lambda^* M$, we have the following adjoint relationships between the wedge and interior products:
\begin{align*}
    \langle \iota_{\gamma} \alpha,\beta\rangle &= \langle \alpha,\gamma \wedge \beta\rangle \\
    \langle \alpha,\iota_{\delta} \beta\rangle &= \langle \delta \wedge \alpha,\beta\rangle.
\end{align*}
Now, the left-hand side of \eqref{DeltaFormulaReduced1} paired with $\alpha$ can be written as
\begin{align*}
    \langle \alpha ,  (-1)^{\deg(\omega)+1} \iota_{\iota_n \omega}(\Omega)\rangle = \langle  (-1)^{\deg(\omega)+1}\iota_n (\omega) \wedge \alpha ,  \Omega\rangle,
\end{align*}
while the right-hand side paired with $\alpha$ can be written as
\begin{align*}
    \langle \alpha , n\wedge \iota_{\omega} (\Omega)\rangle = \langle  \omega \wedge \iota_n(\alpha) ,  \Omega\rangle.
\end{align*}
So now it suffices to check that $(-1)^{\deg(\omega)+1}\iota_n (\omega) \wedge \alpha = \omega \wedge \iota_n(\alpha) $.  Since $\deg(\alpha)+\deg(\omega)=r+1$, we have $\alpha \wedge \omega = 0$, hence
\begin{align*}
    0=\iota_n(\omega\wedge \alpha)=\iota_n(\omega) \wedge \alpha +(-1)^{\deg(\omega)}\omega \wedge \iota_n(\alpha). 
\end{align*}
The claim follows.  
\end{proof}

\subsubsection{$L$-infinity structure}

We next show that the $L$-infinity Jacobi identities hold for the brackets $l_k$ on $A_0$.  Given $i,j\ge0$, an $i$-$j$-unshuffle is a permutation of $1,\ldots,(i+j)$ that preserves the order of $1,\ldots,i$ as well as of $i+1,\ldots,i+j$.
Let $\UnShuff(i,j)$ denote the set of $i$-$j$-unshuffles. 
Recall the notion of the graded signature $\chi$ from \S \ref{GradedSymSection}.

 Fix $i,j,k\in \bb{Z}_{\geq 1}$ such that $i+j=k+1$.  For $\zeta_{\ell}=z^{n_{\ell}}\alpha_{\ell}$, $\ell=1,\ldots,k$, define
\begin{align}\label{Dij}
   D_{ij}\coloneqq  \sum_{\sigma\in \UnShuff(i,j)} \chi(\sigma,\zeta_1,\ldots,\zeta_k)(-1)^{i(k-i)}l_j(l_i(\zeta_{\sigma(1)},\ldots,\zeta_{\sigma(i)}),\zeta_{\sigma(i+1)},\ldots,\zeta_{\sigma(k)}).
\end{align}
The level-$k$ $L$-infinity Jacobi identity states that 
\begin{align}\label{Jacobi}
    \sum_{\substack{i,j\in \bb{Z}_{\geq 1} \\ i+j=k+1}} D_{ij}=0.
\end{align}
For fixed $\sigma \in \UnShuff(i,j)$, the $\epsilon$-factor that appears when applying $l_i$ in \eqref{Dij} is
$$\epsilon(\zeta_{\sigma(1)},\ldots,\zeta_{\sigma(i)})=(-1)^{\sum_{\ell=1}^i (i-\ell)|\zeta_{\sigma(\ell)}|},$$
and the $\epsilon$-factor from when applying $l_j$ is 
$$\epsilon(l_i(\zeta_{\sigma(1)},\ldots ,\zeta_{\sigma(i)}),\zeta_{\sigma(i+1)},\ldots,\zeta_{\sigma(k)})=(-1)^{(j-1)\left[\left(\sum_{\ell=1}^i |\zeta_{\sigma(\ell)}| \right)+(i-2)\right]+\sum_{\ell=1}^{j-1} (j-\ell-1)|\zeta_{\sigma(\ell+i)}|}.$$ 
It follows (keeping in mind that $j-1=k-i$) that $$\epsilon(\zeta_{\sigma(1)},\ldots,\zeta_{\sigma(i)})\epsilon(l_i(\zeta_{\sigma(1)},\ldots ,\zeta_{\sigma(i)}),\zeta_{\sigma(i+1)},\ldots,\zeta_{\sigma(k)}) = (-1)^{i(k-i)}\epsilon(\zeta_{\sigma(1)},\ldots,\zeta_{\sigma(k)}).$$
Hence,
\begin{align}\label{Dijl1}
D_{ij}=l_1\left(\sum_{\sigma\in \UnShuff(i,j)} \chi(\sigma,\zeta_1,\ldots,\zeta_k)\epsilon(\zeta_{\sigma(1)},\ldots,\zeta_{\sigma(k)})l_1(\zeta_{\sigma(1)}\cdots \zeta_{\sigma(i)})\zeta_{\sigma(i+1)}\cdots \zeta_{\sigma(k)}\right).
\end{align}
Now let $n\coloneqq \sum_{\ell=1}^k n_{\ell}$, and for $\sigma\in \UnShuff(i,j)$, let $n_{\sigma}\coloneqq \sum_{\ell=1}^i n_{\sigma(\ell)}$.  Then  \eqref{Dijl1} can be written as
\begin{align}\label{DijIota}
 D_{ij}= z^n\iota_n\left( \sum_{\sigma\in \UnShuff(i,j)} \chi(\sigma,\zeta_1,\ldots,\zeta_k)\epsilon(\zeta_{\sigma(1)},\ldots,\zeta_{\sigma(k)})\iota_{n_{\sigma}}(\alpha_{\sigma(1)} \wedge\cdots\wedge \alpha_{\sigma(i)})\wedge \alpha_{\sigma(i+1)} \wedge\cdots\wedge \alpha_{\sigma(k)} \right).
\end{align}
Since contractions are anti-derivations, we have
\begin{equation}
 \begin{array}{l}
    \iota_{n_{\sigma}}(\alpha_{\sigma(1)} \wedge\cdots\wedge \alpha_{\sigma(i)})\\[2mm] 
    \quad= \sum_{\ell=1}^i \left((-1)^{\sum_{q=1}^{\ell-1} \deg(\alpha_{\sigma(q)})}\right) \alpha_{\sigma(1)}\wedge \cdots \wedge \alpha_{\sigma(\ell-1)} \wedge \iota_{n_{\sigma}}( \alpha_{\sigma(\ell)}) \wedge \alpha_{\sigma(\ell+1)} \wedge \cdots \wedge \alpha_{\sigma(i)}.
    \end{array}
    \label{eq-anti-derivation}
\end{equation}
 Combining this with \eqref{chiomega}, we find that 
\begin{align}\label{sigmaFixed}
    \chi(\sigma,\zeta_1,\ldots,\zeta_k)&\epsilon(\zeta_{\sigma(1)},\ldots,\zeta_{\sigma(k)})\iota_{n_{\sigma}}(\alpha_{\sigma(1)} \wedge\cdots\wedge \alpha_{\sigma(i)})\wedge \alpha_{\sigma(i+1)} \wedge\cdots\wedge \alpha_{\sigma(k)} \nonumber \\
    =& \epsilon(\zeta_1,\ldots,\zeta_k)\sum_{\ell=1}^i \left((-1)^{\sum_{q=1}^{\sigma(\ell)-1} \deg(\alpha_q)}\right) \alpha_1\wedge \cdots \wedge \iota_{n_{\sigma}} (\alpha_{\sigma(\ell)}) \wedge \cdots \wedge \alpha_k.
\end{align}
Here, for $S_{\ell}\coloneqq \{1,2,\ldots,\sigma(\ell)-1\}\setminus \{\sigma(1),\sigma(2),\ldots,\sigma(\ell-1)\}$, the contribution of $\sum_{q\in S_{\ell}}\deg(\alpha_q)$ to the exponent of $(-1)$ is the result of the difference between moving $\iota_{n_{\sigma}}\alpha_{\sigma(\ell)}$ past $\alpha_q$ for each $q\in S_{\ell}$ versus moving $\alpha_{\sigma(\ell)}$ past $\alpha_q$ for each $q\in S_{\ell}$.  Combining \eqref{sigmaFixed} with \eqref{DijIota}, we now have
\begin{align}\label{DijAgain}
    D_{ij}=z^n \epsilon(\zeta_1,\ldots,\zeta_k) \iota_n\left(\sum_{\sigma\in\UnShuff(i,j)} \sum_{\ell=1}^i (-1)^{\sum_{q=1}^{\sigma(\ell)-1} \deg(\alpha_q)} \alpha_1\wedge \cdots \wedge \iota_{n_{\sigma}} (\alpha_{\sigma(\ell)}) \wedge \cdots \wedge \alpha_k \right)
\end{align}

Note that the data an unshuffle $\sigma\in \UnShuff(i,j)$ is equivalent to the data of an $i$-element subset of $\{1,\ldots,k\}$, where this choice of subset is identified with $\{\sigma(1),\ldots,\sigma(i)\}$. So
given distinct $s,t\in \{1,\ldots,k\}$, there are exactly $\binom{k-1}{i-1}$ unshuffles $\sigma \in \UnShuff(i,j)$ for which $s\in \{\sigma(1),\ldots,\sigma(i)\}$, and exactly $\binom{k-2}{i-2}$ unshuffles $\sigma \in \UnShuff(i,j)$ for which $s$ and $t$ are both in $\{\sigma(1),\ldots,\sigma(i)\}$. Thus, for each $t\in \{1,\ldots,k\}$, the expression
\begin{align}\label{BigParenthesesPart}
    \sum_{\sigma\in\UnShuff(i,j)} \sum_{\ell=1}^i (-1)^{\sum_{q=1}^{\sigma(\ell)-1} \deg(\alpha_q)} \alpha_1\wedge \cdots \wedge \iota_{n_{\sigma}} (\alpha_{\sigma(\ell)}) \wedge \cdots \wedge \alpha_k
\end{align}
from the large parentheses of \eqref{DijAgain} includes precisely $\binom{k-1}{i-1}$ terms of the form $$(-1)^{\sum_{q=1}^{s-1} \deg(\alpha_q)} \alpha_1\wedge \cdots \wedge \iota_{n_{\sigma}} (\alpha_{s}) \wedge \cdots \wedge \alpha_k$$
for various $n_{\sigma}$.  Furthermore, $n_s$ is a term in $n_{\sigma}$ for all $\binom{k-1}{ i-1}$ of these terms, while $n_t$ for $t\neq s$ is a term in $n_{\sigma}$ for $\binom{k-2}{i-2}$ of these terms.  Thus, the expression from \eqref{BigParenthesesPart} is equivalent to
\begin{align*}
    \binom{k-2}{i-2} \iota_n(\alpha_1\wedge \cdots \wedge \alpha_k)  + \left(\binom{k-1}{i-1} - \binom{k-2}{i-2}\right) \sum_{\ell=1}^k (-1)^{\sum_{q=1}^{\ell-1} \deg(\alpha_q)} \alpha_1 \wedge \cdots \wedge \iota_{n_{\ell}}\alpha_{\ell} \wedge \cdots \wedge \alpha_k.
\end{align*}
Substituting into the large parentheses from \eqref{DijAgain}, and then using the fact that $\iota_n\circ \iota_n=0$ along with the Pascal's triangle identity $\binom{k-2}{i-2} + \binom{k-2}{i-1}= \binom{k-1}{i-1}$, we obtain
\begin{align*}
    D_{ij}=z^n\epsilon(\zeta_1,\ldots,\zeta_k)\binom{k-2}{i-1} \iota_n \left(\sum_{\ell=1}^k (-1)^{\sum_{q=1}^{\ell-1} \deg(\alpha_q)} \alpha_1 \wedge \cdots \wedge \iota_{n_{\ell}}\alpha_{\ell} \wedge \cdots \wedge \alpha_k\right).
\end{align*}
 Hence,
\begin{align*}
    \sum_{i+j=k+1} D_{ij} = 2^{k-2}z^n\epsilon(\zeta_1,\ldots,\zeta_k) \iota_n \left(\sum_{\ell=1}^k (-1)^{\sum_{q=1}^{\ell-1} \deg(\alpha_q)} \alpha_1 \wedge \cdots \wedge \iota_{n_{\ell}}\alpha_{\ell} \wedge \cdots \wedge \alpha_k\right)
\end{align*}

This shows that the Jacobi identities in general fail on $A$ (except for the $k=1$ case, which just says $l_1\circ l_1=0$).  However, if $\iota_{n_{\ell}}\alpha_{\ell}=0$ for each $\ell$, then we in fact get that each $D_{ij}=0$.  Combining this with \eqref{degree k-2} and \eqref{GradedSkew}, we have shown that $A_0$ is an $L$-infinity algebra:

\begin{thm}\label{Linfinity}
The brackets $l_k$ make $A_0$ into an $L$-infinity algebra.  Furthermore, each $D_{ij}$ as in \eqref{Dij} equals $0$ on $A_0$.
\end{thm}

\begin{rmk}
We have that $l_1^2=0$ on all of $A$, and we have seen that $l_2|_{A_0}$ extends to a Lie bracket (the Schouten-Nijenhuis bracket) on all of $A$.  We therefore wonder whether the other brackets $l_k|_{A_0}$ admit extensions which make $A$ into an $L$-infinity algebra.  One exotic $L$-infinity structure on polyvector fields was constructed in \cite{Shoi}, but this evidently does not restrict to our $L$-infinity structure since the bracket of loc. cit. is trivial in odd degree, whereas our $l_k$ is non-trivial on $A_0$ for all $k\geq 2$.  We also wonder whether the closed elements of a BV-algebra might more generally admit an $L$-infinity structure via a construction analogous to ours, i.e., with the $k$-bracket being given up to sign by taking the associative product and then applying the BV-operator.
\end{rmk}

\subsection{Wall-crossing acts on polyvector fields via the Schouten-Nijenhuis bracket}\label{MSapps}
In the Gross-Siebert program, the construction of a ``mirror'' space proceeds by formally gluing together a collection of algebraic tori via certain wall-crossing automorphisms.  These automorphisms can be viewed as follows.

Fix a commutative ring $R$, and let $P\subset N$ denote the set of integral points of a strictly convex cone in $N_{\bb{R}}$.  Consider the following ring of Laurent series of polyvector fields on the algebraic torus $\Spec R[N]$: 
$$\wh{A}\coloneqq \bigcup_{n\in N}z^{n} R\llb P \rrb\otimes \Lambda^* M$$ where $R\llb P\rrb$ is the power series ring obtained by completing $R[P]$ with respect to its unique monomial maximal ideal.  Let $\wh{A}^d$ denote the $d$-graded part of $\wh{A}$ under our grading $|\cdot|$.  In particular, $\wh{A}^{-1}$ is the Laurent series ring $R(\!(P)\!)\coloneqq \bigcup_{n\in N} z^n R\llb P \rrb$.  Note that our brackets $l_k$ extend to $l_k:\wh{A}^{\otimes k}\rar \wh{A}$.  In particular, we can define $\wh{A}_0\coloneqq \ker(l_1)\subset \wh{A}$ and $\wh{A}_0^d=\wh{A}_0\cap \wh{A}^d$.

The algebra $\wh{A}^0$ is a Lie algebra, what \cite{GPS} calls the module of log derivations, with the bracket being the usual bracket of vector fields.  Let $\f{g}\coloneqq \wh{A}_0^0$.  This forms a Lie subalgebra of $\wh{A}^0$, and since the usual Lie bracket of vector fields agrees with the restriction of the Schouten-Nijenhuis bracket, Proposition \ref{Schouten} implies that we can view $\wh{A}_0^0$ as a Lie algebra under our bracket $l_2$.  
The corresponding Lie group $G\coloneqq \exp(\f{g})$ is what \cite{GPS} calls the \textbf{tropical vertex group}.

We next describe the transformations that generate this group, cf. (2.19) in \cite{GHS}.
We note that for $\alpha\in \widehat A$ and $f\in R(\!(P)\!)=\wh{A}_0^{-1}=\wh{A}^{-1}$, we have the action
\begin{equation}
\alpha(f)\coloneqq \iota_{df}\alpha=[\alpha,f]=\ad_{\alpha}(f).
\label{eq-adjoint-action-on-functions}
\end{equation}
Given $n\in P$, let $\f{g}_n^{\parallel}\subset \f{g}$ denote the Lie subalgebra spanned (topologically) by elements of the form $z^{kn}\otimes m$, $k\in \bb{Z}_{\geq 1}$ and $m\in M$.  Recall as in \cite[Def. 2.11]{GHS} that a ``wall $\f{d}$ with direction $-n$'' is a polyhedral subset of $N_\RR$ and has attached to it an element of $g_{\f{d}}\in \f{g}_n^{\parallel}$. The associated wall-crossing automorphism of $R(\!(P)\!)$ is given by $\exp g_{\f{d}}$ (viewing $g_{\f{d}}$ as a derivation acting on $\wh{A}^{-1}$).  Equivalently, we may view $\exp g_{\f{d}}$ as $\exp \ad_{g_{\f{d}}}$, the restriction of the adjoint action of $\wh{A}$.  For example, for $p\in P$ and $f$ of the form $1+\sum_{k=1}^{\infty} a_kz^{kp}$ with $a_k\in R$ and $m\in p^{\perp}\subset M$, we have the $\ell$'th iterated application $(\ad_{\log(f)\otimes m})^\ell z^n = (\log(f^{\langle n,m\rangle}))^\ell z^n$
via \eqref{eq-adjoint-action-on-functions} and then one finds
\begin{align*}
    (\exp \ad_{\log(f)\otimes m})z^n = z^nf^{\langle n,m\rangle}.
\end{align*}
In the Gross-Siebert construction, one has a scattering structure consisting of walls $(\f{d},g_{\f{d}})$.  Each chamber corresponds to a copy of $\Spec R\llb P\rrb$, and gluing all these schemes together via wall-crossing automorphisms $\exp \ad_{g_{\f{d}}}$ yields (a dense open subset of) the mirror space $\s{X}$.

Of course, rather than restricting each $\exp \ad_{g_{\f{d}}}$ to just $\wh{A}_0^0$, we can consider the action on all of $\wh{A}_0$, which we note can be viewed as the space of polyvector fields on $\wh{A}_0^0$ which are closed under the operator $\Delta$ of \S \ref{l1Delta}.  Even better, we can extend the action to the space $\wh{A}$ of all polyvector fields using the Schouten-Nijenhuis bracket $[\cdot,\cdot]$ as in Proposition \ref{Schouten}.  We refer to these as \textbf{extended wall-crossing automorphisms}.

On the other hand, if we are interested in gluing polyvector fields from different copies of $\Spec R\llb P\rrb$, then we must understand how the automorphism $\exp \ad_{g_{\f{d}}}$ of $\wh{A}_0^0$ acts on these polyvector fields via pushforward, i.e., the multivector derivative $(\exp \ad_{g_{\f{d}}})_*$ of $\exp \ad_{g_{\f{d}}}$.  The following proposition says that these two actions are the same.
\begin{prop}\label{dl2}
For any $\alpha \in \wh{A}$, $$(\exp \ad_{g_{\f{d}}})_* (\alpha) = (\exp \ad_{g_{\f{d}}}) (\alpha).$$  Hence, gluing elements of $\wh{A}$ via our extended wall-crossing automorphisms produces meromorphic sections of the sheaf of polyvector fields on $\s{X}$.
\end{prop}
\begin{proof}
We can assume $\alpha$ is homogeneous with $\deg(\alpha)=d$.  We proceed by induction on $d$, noting that the $d=0$ case is trivial.

A homogeneous multivector field $\alpha\in \wh{A}$ of positive degree is characterized by its action on functions $f\in R\llb P\rrb$ via $\alpha(f)=\iota_{df}(\alpha)=\ad_{\alpha}(f)$. 
   The pushforward action $[(\exp \ad_{g_{\f{d}}})_* (\alpha)](f)$ is then given by using $(\exp \ad_{g_{\f{d}}})^{-1}$ to pull back $f$, acting on $f$ by $\alpha$, and then pushing the resulting degree-$(d-1)$ polyvector field forward using $(\exp \ad_{g_{\f{d}}})_*$, which by the inductive assumption is the same as applying $\exp \ad_{g_{\f{d}}}$.  I.e.,
   \begin{align*}
           (\exp \ad_{g_{\f{d}}})_* (\alpha) &= \exp (\ad_{g_{\f{d}}}) \alpha \exp (-\ad_{g_{\f{d}}}) \\
    &= \Ad_{\exp (\ad_{g_{\f{d}}})}(\alpha) \\
    &= [\exp \ad_{\ad_{g_{\f{d}}}}](\alpha).
    \end{align*}
Since $\alpha$ acts on $f\in R\llb P\rrb$ as $\ad_{\alpha}$, the above expression acts on $f$ as $[\exp \ad_{\ad_{g_{\f{d}}}}](\ad_{\alpha})$.  Since $\ad_{\ad_a}(\ad_b)=[\ad_a,\ad_b]=\ad_{[a,b]}$, it follows by induction that $\ad^k_{\ad_a}(\ad_b)=\ad_{\ad_a^k(b)}$, and so $$\exp \ad_{\ad_a}(\ad_b) = \ad_{\exp \ad_a(b)}.$$  We thus see that the action of $[\exp \ad_{\ad_{g_{\f{d}}}}](\alpha)$ on $f$ is via $\ad_{\exp \ad_{g_{\f{d}}}(\alpha)}$, meaning that $[\exp \ad_{\ad_{g_{\f{d}}}}](\alpha)=(\exp \ad_{g_{\f{d}}})(\alpha)$, as desired.
\end{proof}

\subsection{Theta polyvector fields and mirror symmetry}\label{ThetaPolyvectMS}
One can construct canonical bases of ``theta functions'' on Gross-Hacking-Keel-Siebert mirror spaces, cf. \cite{GHK1,GHKK,GHS}.  These theta functions can be expressed in terms of certain counts of tropical disks and tropical curves, cf. \cite{CPS,Man3}.  The original motivation for this article was to show that the tropical multiplicities used for the counts in \cite{Man3} agree with the multiplicities of \cite{MRud} and thus give log Gromov-Witten invariants.  Indeed, this follows from Theorem \ref{BracketMult}, cf. Example \ref{ThetaEx}, and found application in \cite{ManFrob} to prove that the Frobenius structure conjecture of \cite[\S 0.4, arXiv v1]{GHK1} holds for cluster varieties.  Roughly, this conjecture claims that the theta functions can be described in terms of certain descendant log Gromov-Witten invariants.

Proposition \ref{dl2} suggests to us that similar methods can be used to show much more: according to closed string mirror symmetry, the symplectic cohomology of a log Calabi-Yau variety with affine mirror should be isomorphic to the ring of polyvector fields on the mirror (cf. \cite[\S 1]{Pas}).  In particular, $SH_{\log}^0(Y,D)$ is expected to be the coordinate ring of the mirror, i.e., it should be spanned by the theta functions.  As noted in \S \ref{Speculation}, a construction of the log symplectic cohomology ring from punctured GW invariants is being investigated by Gross-Pomerleano-Siebert \cite{GPomS}, building off the ideas of Ganatra-Pomerleano \cite{GP,GP2}.

A construction of ``theta polyvector fields,'' analogous to the construction of theta functions, is not currently known, but we suggest based on Proposition \ref{dl2} that such a construction might be possible using the Schouten-Nijenhuis bracket and higher-codimension tropical incidence conditions, at least for the BV-closed polyvector fields $A_0$.
\begin{rmk}
We note that our brackets $l_k$ can be defined for more general log Calabi-Yau varieties $U$ via $$l_k(\zeta_1,\ldots,\zeta_k)\coloneqq \epsilon(\zeta_1,\ldots,\zeta_k)l_1\left(\prod_{j=1}^k \zeta_j \right)$$ by taking $l_1$ to be the BV--operator as in \S \ref{l1BV}, and defining $\epsilon$ as in \eqref{epsilon}.  We expect these brackets form an $L_{\infty}$-structure on this more general $A_0$ --- indeed, if $U$ contains a Zariski dense algebraic torus, then this follows from the torus case, i.e., Theorem \ref{Linfinity}.  As noted above, for $Y\setminus D$ affine, closed string mirror symmetry predicts that $SH_{\log}^*(Y,D)\cong \Poly^*(U)$ for $U=\Spec SH_{\log}^0(Y,D)$.  We conjecture that this isomorphism identifies our brackets $\{l_k\}_{k\in \bb{Z}_{\geq 1}}$ on $A_0$ with the $L_{\infty}$-structure of Chas-Sullivan \cite{CS} on the equivariant string topology of $Y\setminus D$ --- cf. Observation \ref{Janko2} for the case where $Y\setminus D$ is an algebraic torus.
\end{rmk}

\begin{appendix}
\section{Relation to String Topology and Symplectic Cohomology}\label{app}
We learned the following in conversations with Janko Latschev.
Let $N,M$ be dual lattices of rank $r$ as before. We set $T=N_\RR/N \cong (S^1)^r$, so that $N=H_1(T,\ZZ)$.
Notice that $N$ indexes the free homotopy classes of loops in $T$ and hence also the components of the free loop space $\s LT$ of $T$.
In this special situation, the evaluation map $\s LT\ra T$ at the base point is a homotopy equivalence on each connected component of $\s LT$. In particular each component has the homology of $T$, which as a graded $\ZZ$-module can be identified with $\bigwedge^\bullet N$. Choosing an orientation of $T$, i.e. a generator {$\Omega$} of $\bigwedge^rM$, gives rise to an isomorphism
\begin{align*}
  \bigwedge {}\!^{r-\bullet} N &\cong \bigwedge {}\!^{\bullet} M \\
  \alpha&{\mapsto \iota_{\alpha}\Omega},
\end{align*}
and so  we get the identification of graded $\ZZ$-modules 
\begin{equation} \label{iso-loop-module}
H_{r-\bullet}(\s LT,\ZZ)\cong \ZZ[N]\otimes_\ZZ\bigwedge{}\!^\bullet M\equiv A.
\end{equation} 
The homology of $\s LT$ carries the \emph{loop product} \cite{CS} which combines concatenation of loops and intersection theory in the base manifold and has geometric degree $-r$, so it gives a degree 0 map
\begin{equation} \label{loop-product}
H_{r-\bullet}(\s LT)\otimes H_{r-\bullet}(\s LT)\ra H_{r-\bullet}(\s LT).
\end{equation}
$H_\bullet(\s LT)$ also carries a BV-operator which comes from the $S^1$-action of moving the base point of the loop,
\begin{equation} \label{loop-BV}
H_\bullet(\s LT)\ra H_{\bullet+1}(\s LT)
\end{equation}
turning the homology of the free loop space into a BV-algebra. Simple geometric considerations now yield

\begin{obs} \label{Janko1}
Under the isomorphism \eqref{iso-loop-module}, the loop product \eqref{loop-product} gets identified with the usual product in $A$ given by 
$(z^{n_1}\otimes\alpha_1)\cdot (z^{n_2}\otimes\alpha_2)=z^{n_1+n_2}\otimes (\alpha_1\wedge\alpha_2)$.
Furthermore, the BV-operator \eqref{loop-BV} becomes the operator $l_1:z^{n}\otimes\alpha\mapsto z^{n}\otimes\iota_n\alpha$. 
\end{obs}

The Viterbo isomorphism \cite{Viterbo} identifies the symplectic cohomology of the cotangent bundle $T^*T$ with the homology of the free loop space, so in standard grading conventions we have
\begin{equation} \label{SHtoHL}
\SH^\bullet(T^*T,\ZZ) \cong H_{r-\bullet}(\s LT,\ZZ),
\end{equation}
which is an isomorphism of BV-algebras by \cite[Corollary 6.1.2]{Abouzaid}. 
Combining this with Observation~\ref{Janko1} gives:
\begin{obs} \label{Abouzaid}
$A\cong \SH^\bullet(T^*T,\ZZ)$ as BV-algebras.
\end{obs}
Note that $T^*T$ is the mirror dual to the algebraic torus $M\otimes\bb{C}^*$.

We finally turn to the $L_\infty$-structure on $A_0=\ker(l_1)$. 
There is a natural map from $S^1$-equivariant to ordinary homology $H^{S^1}_\bullet(\s LT) \ra H_{\bullet+1}(\s LT)$, which is called ``mark'' in \cite{CS}.  In the case of the torus, the kernel of this map consists precisely of the homology of the component of contractible loops, and so we get an injection
\begin{equation}\label{MarkMap}
H_{r-1-\bullet}^{S^1}(\s LT,T) \ra H_{r-\bullet}(\s LT).
\end{equation}
$S^1$-equivariant homology of the loop space carries its own Lie bracket, known as the {\em string bracket} \cite[Theorem 6.1]{CS}, and it is easy to see that the map \eqref{MarkMap} is a morphism of Lie algebras, where on $H_{r-\bullet}(\s LT)$ we use the Lie bracket of degree 1 induced from the BV-operator known as the {\em loop bracket}.

Chas and Sullivan also described an $L_\infty$-structure on the equivariant homology of a free loop space which is built from the product and the BV-operator in ordinary homology of $\s LT$ and satisfies even stronger relations than usually required \cite[Theorem 6.2]{CS}, in our notation $D_{ij}+D_{ji}=0$.
It turns out that under the identification in Observation~\ref{Janko1}, this $L_\infty$-structure is precisely the one (re-)discovered in Theorem~\ref{Linfinity}. 
Note that in Theorem~\ref{Linfinity} we proved $D_{ij}=0$ which is yet slightly stronger a condition than what Chas-Sullivan found.
\begin{obs} \label{Janko2}
The ``mark'' map $H^{S^1}_{r-1-\bullet}(\s LT) \ra H_{r-\bullet}(\s LT)\cong A$ has $A_0\setminus \ZZ=\ker (l_1)\setminus \ZZ$ as its image (where by $\bb{Z}$ we mean $\bb{Z}\otimes \Lambda^0 M\subset \bb{Z}[N]\otimes \Lambda^* M$) and the homology of the component of contractible loops as its kernel. Moreover, it is a morphism of $L_\infty$-algebras where on the $S^1$-equivariant homology of $\s LT$ we use the $L_\infty$-structure of \cite[Theorem 6.2.]{CS} and on the submodule $A_0\subset A$ we use the $L_\infty$-structure described in \S\ref{Linf}.
\end{obs}

In \cite{Tonk}, Tonkonog considers certain $L_{\infty}$-augmentations of $CF^*_{S^1,+}(M)$ where $M$ is a Liouville domain (e.g., $M=T^*T$) and $CF^*_{S^1,+}$ denotes the positive equivariant Floer complex equipped with the Chas-Sullivan $L_{\infty}$-bracket.  The augmentations are defined in terms of gravitational descendant invariants which are closely related to the descendant log Gromov-Witten invariants we consider --- it  appears that Tonkonog's augmentations, at least for $M=T^*T$, are mirror to our $l_k$'s times a factor which counts permutations of the markings modulo automorphisms of the underlying genus $0$ curve.

\end{appendix}
   
\bibliographystyle{alpha}  
\bibliography{biblio}        
\index{Bibliography@\emph{Bibliography}}%

\end{document}